\documentclass[10pt]{amsart}
  
\input xy
\xyoption{all}
  
\usepackage{amssymb}
\usepackage{amsmath}
\usepackage{amscd}
\usepackage{latexsym}
\usepackage{bbold}
\usepackage{esvect}
\usepackage{graphicx}
\usepackage{mathrsfs}  
  
\newtheorem{lemma}{Lemma}[section]
\newtheorem{prop}[lemma]{Proposition}
\newtheorem{cor}[lemma]{Corollary}
  
\newtheorem{thm}[lemma]{Theorem}

\theoremstyle{definition}
  
\newtheorem{rem}[lemma]{Remark}
\newtheorem{rems}[lemma]{Remarks}
\newtheorem{defi}[lemma]{Definition}

\newcommand{\A}{\measuredangle}

\newcommand{\Proba}{{\mathrm{Proba}}}
\DeclareMathOperator{\Isom}{Isom}
\DeclareMathOperator{\Aff}{Aff}
\DeclareMathOperator{\Co}{Cone}
\DeclareMathOperator{\Cay}{Cay}

\title{Proper actions on $\ell^p$ spaces for relatively hyperbolic groups}
\author{Indira Chatterji and Fran\c cois Dahmani}

\begin{document}
 
\maketitle
\begin{abstract}
We show that for any group $G$ that is hyperbolic relative to subgroups that admit a
proper  affine isometric action on a uniformly convex Banach space, then $G$ acts properly on a uniformly convex Banach space as well.
\end{abstract}

\section*{Introduction}
Kazhdan's property (T) was introduced in 1967 by D. Kazhdan in
\cite{Kaz} and since then has been intensively studied. A
reformulation of Kazhdan property (T), following from work of
Delorme-Guichardet, is that any action on a Hilbert space has a fixed
point. Groups with property (T) include higher rank lattices in simple
Lie groups, such as $SL(n,{\bf Z})$ for $n\geq 3$, but also some
lattices in rank one Lie groups, such as lattices in $Sp(n,1)$ for
$n\geq 3$. The uniform lattices in $Sp(n,1)$ being hyperbolic, it
shows that some hyperbolic groups can have property (T), as opposed to
some others, like free groups or surface groups, that admit a proper
action on a Hilbert space. The latter is called Haagerup property or
aTmenability. The first natural generalization of a Hilbert space are
$L^p$ space and it is implicit in P. Pansu's results in \cite{Pan}, that lattices in $Sp(n,1)$ admit an action without fix points on an $L^p$ space. G. Yu in \cite{Yu} shows that in fact any hyperbolic group admits a proper action on an $\ell^p$-space for $p$ large enough. M. Bourdon in \cite{Bou}, B. Nica in \cite{Ni} and more recently A. Alvarez and V. Lafforgue in \cite{AL} gave an alternative proof of Yu's theorem. On the contrary, higher rank lattices are known to not act on any $L^p$ or uniformly convex Banach space (\cite{BFGM} and \cite{La}).

Relatively hyperbolic groups are a  geometric generalization of hyperbolic groups and have been studied a lot. They allow  interesting group constructions while retaining a lot of the geometry of hyperbolic spaces. Relatively hyperbolic groups can also have property (T), but their
peripheral subgroups can forbid any action on a uniformly convex Banach space. We show that this is in fact the only obstruction. Working with Alvarez-Lafforgue's construction, we prove that any
relatively hyperbolic group has a non-trivial affine isometric action on an $\ell^p$
space, for some $p$ large enough. If furthermore its peripheral subgroups act properly by affine isometries on some
uniformly convex Banach space, then the whole group also acts properly on a uniformly convex Banach space. It is the case for instance for Haagerup peripheral subgroups, in particular for all amenable ones. Our main result is the following.
\begin{thm}\label{main} Let $G$ be a relatively hyperbolic group, non-elementary. Then, for
$p$ sufficiently large $G$ admits an isometric action on an
$\ell^p$-space with an unbounded orbit.

Moreover, if each of the parabolic subgroup act properly on a uniformly convex Banach space (respectively, on an $\ell^p$ space), then so does $G$ (respectively, $G$ acts on an $\ell^q$ space for $q$ large enough).
\end{thm}
Even if they do not appear explicitly in print, some parts, and particular cases of this statement were known to experts. For instance, putting together  results of Puls on $L^p$-cohomology \cite{Puls}, and of Gerasimov \cite{Gera}, on the Floyd boundary of relatively hyperbolic groups, one can show that non-elementary relatively hyperbolic groups act on some $L^p$-space without fixed point.  E. Guentner, E. Reckwerdt and R. Tessera obtained a version of Theorem \ref{main} for groups that are relatively hyperbolic with virtually nilpotent parabolic subgroups \cite{GRT} and recently, E. Guentner and R. Tessera announced in talks the case of amenable parabolic subgroups.

Let us comment on the general strategy, from  Alvarez-Lafforgue in
\cite{AL}, Yu \cite{Yu} and Mineyev \cite{Min}. Given a representation
$\pi$ of a group $G$ in the isometry group of a normed vector space
$V$, one can obtain an affine isometric action by constructing a
cocycle $c:G\to V$. The action is given by $g\cdot v =
\pi(g)(v)+c(g)$. 

For instance, consider $W= \ell^1(G, \mathbb{R})$, and $V_\infty= \ell^\infty(G,W)$.
Then consider the
representation $\pi:G\to {\rm Isom} (V_\infty)$  by precomposition of the maps in
$\ell^\infty$. If $G$ is discrete,  for all $p$, $V_p= \ell^p (G,W)$ is a
subspace of  $V_\infty$,  preserved by $\pi$, and on which $\pi(G)$ consists of isometries.   For
an affine action on $V_p$, one thus needs a cocycle $c:G\to
V_\infty$ so that, for all $G$, $c(g)$ is $\ell^p$-summable
(in other words, for all $g$, the sum $\sum_{h\in G} \| (c(g))(h)\|_W^p$
must be finite). 

A traditional example is when
$G$ is free. In that case,  let $d\in V_\infty $ be such that $d(h)$ is the
indicator of the unique neighbor of $h$ closer to $1$  in the word
metric (and $d(1)=0$).  One may see this as the arrival point at $h$ of the geodesic flow from $1$.  Then $c(g) =d-\pi(g)d$ is also in $V_\infty$, and is supported by the interval $[1,g]$ in 
the Cayley tree of $G$. Since it has finite support, it is in $V_2 = \ell^2(G,W)$. The map $c$ is  from $G$ to  $V_2$, and it is     a cocycle for the representation $\pi$, because it is  the 
coboundary of $d$  in
$V_\infty  $ (which makes the formal cocycle relation hold).
 One thus obtains an affine action of the free group on $\ell^2 (G,W)$. The knowledgable reader may have noticed that this is a variation on the more usual action on $\ell^2(E,\mathbb{R})$ where $E$ is the set of oriented edges of a Cayley tree of the free group. In this more classical (arguably simpler)  action, the function $d$ is the indicator of the set of all oriented edges pointing toward $1$ (those whose end point is closer to $1$ than their starting point).

 Defining an analoguous cocycle $c$ for other groups is more difficult, sometimes impossible, and is the reason of our seemingly complicated presentation of the free group case. One may opt for trying to determine how much should $h\in G$ be thought in the interval $[1,g]$. For this, one may use a geodesic flow, and test whether the flow from $1$ and the flow from $g$ toward $h$ tend to arrive to $h$ similarily or differently. 
Writing $\mu_1(h), \mu_g(h)$ the final values of such  flows (in a suitable normed vector space $W$),  the measurement of this difference  defines $c(g) = (\mu_1 - \mu_g) :G\to W $, a map from $G$ in $W$. It provides $c: G\to \ell^p(G, W)$ if the confluence is fast enough (exponential) for almost all $h\in G$, and if $p$ is choosen sufficiently large to compensate the exponential growth of $G$.

Our argument is then through the construction of a flow as done by Alvarez-Lafforgue in \cite{AL}. This flow takes an initial point, and a target point (in
the coned-off Cayley graph, or more generally in any fine $\delta$-hyperbolic graph), and flows
the initial point toward the target point, using a family of
``reasonable'' paths: those were ``$\alpha$-geodesics'' in the work of Alvarez-Lafforgue \cite{AL} , but to handle the lack of local finiteness we need {\it conical $\alpha$-geodesics}. During the flow, the position of the flowed
initial point becomes a probability measure, called the {\it mask of the
target point for the initial point} (that is what the target point needs
to mask the initial point from its sight and is reminiscent of the notion of wall). The whole argument is to control
stability of the flow (along the way, it won't escape an a-priori controlled zone) and
confluence of the flow (two sources close to each other, with same target, will flow to
the same positions with positive probability at each step).  In a hyperbolic group, it is done
by Alvarez-Lafforgue in \cite{AL}, and relies on the uniform local finiteness. Here, we
work in highly non-locally finite setting. Using the tools developed to handle
this non-local finiteness in \cite{Dthesis}, we manage to construct a flow. The key here is that, sometimes,
due to the angular geometry of relative hyperbolicity,  
the flow steps are so much confluent that they become equal in one step, which compensates the lack of local finiteness. 
This flow, which we think of as some kind of geodesic flow,
provides a cocycle $c$, so that $c(g)(h)$ is the difference between the masks of $h$ for $1$ and for $g$,  that is $\ell^p$-summable (for $h$ ranging over $G$), hence a nice action. 

The flow
also provides, not coset representatives for peripheral subgroups, but {\it random coset representatives}, which are probability measures for
coset representatives (see Definition \ref{randomCoset}). This allows us to construct a cocycle for the induced
representation, if ever a peripheral subgroup does act. Random coset representatives are not restricted to relatively hyperbolic groups, and for instance in the case of the Haagerup property, we obtain a criteria (Proposition \ref{Haagerup}), which generalizes the fact that free product with amalgamations over finite groups of groups with the Haagerup property have the Haagerup property themselves (and similarly for actions on an $\ell^p$-space, see \cite{Ar}). An interesting corollary is the context of small cancellation groups $C'(\lambda)$ over free products (see \cite[Chapter V.9]{LS77}), for small $\lambda$.
\begin{cor}\label{scH} Let $\lambda <1/6$. A group that is small cancellation $C'(\lambda)$ over a free product of Haagerup groups (respectively, of groups acting properly on an $\ell^p$-space for
  some $p>1$), is itself Haagerup (respectively, acts on an $\ell^p$-space for that same $p$).\end{cor}
\medskip

\noindent
{\it Acknowledgements:} The authors would like to thank Vincent
Lafforgue and Aur\'elien Alvarez for their explanations of their
result, as well as Erik Guentner and Cornelia Drutu for discussions on
the subject. We thank Marc Bourdon for pointing out the relevance of
\cite{Puls} and \cite{Gera} together, and Domink Gruber for pointing
out \cite{GS}. Finally we warmly thank the referee for suggesting
several improvements in the text.  This project started at MSRI during the special semester
in {\it Geometric Group Theory}, and continued during the INI special
semester {\it Non-positive Curvature group actions and cohomology}, supported by EPSRC grant no EP/K032208/1. The authors would like to thank both programs for the great working conditions. Both authors are partially supported by IUF, and the first author by ANR GAMME.
\section{Working in relatively hyperbolic groups}\label{bases}
To fix notations we recall that a {\it graph} $X$ is a set $X^{(0)}$, the {\it vertex set} with a set $X^{(1)}\subseteq X^{(0)}\times X^{(0)}$, the {\it oriented edges}, endowed with an {\it origin map} $o:X^{(1)} \to X^{(0)}$ and a fixed-point free involution reversing the edges $\bar{ }: X^{(1)}\to X^{(1)}$. A pair $\{e,\bar e\}$ is an {\it unoriented edge}. We denote by $t:X^{(1)} \to X^{(0)}$ the composition $\bar o$, this is the {\it terminus map}. We assume the reader familiar with paths, connectedness, length, geodesics, in graphs, as well with hyperbolicity. All our graphs are considered with their graph metric where an edge has length one.
\begin{defi}
Let $X$ be a graph. Given a vertex $v$ and two oriented edges $e_1,
e_2$ such that $o(e_2) = t(e_1) =v$, the {\it angle between $e_1$ and
  $e_2$ at $v$}, denoted by $\A_v(e_1,e_2)$, is the infimum (in
$\mathbb{R_+} \cup \{+\infty\}$)  of length  of paths
from  $o(e_1)$ to $t(e_2) $  that avoid $v$. In other words,
$$\A_v(e_1,e_2)=d_{X\setminus\{v\}}(o(e_1),t(e_2)).$$
By convention, we will use the following abuse of notation: if $x_1, x_2,c$
are vertices in $X$, we say that $\A_c(x_1,x_2)>\theta$ if there
exists two edges $e_1, e_2$ with $t(e_1)= o(e_2) =c$, with  $e_i$ on a geodesic between $c$ and $x_i$ ($i=1,2$) and such
that $\A_c(e_1,e_2)>\theta$. We say that  $\A_c(x_1,x_2)\leq\theta$ otherwise, namely, if for any geodesic between $c$ and $x_i$ and starting by and edge $e_i$($i=1,2$), then $\A_c(e_1,e_2)\leq\theta$. 
\end{defi}
Large angles at a vertex in a $\delta$-hyperbolic graph will force geodesic through that vertex. More precisely we have the following.
\begin{prop}\label{prop;goulet}
Let $X$ be a $\delta$-hyperbolic graph, and let $a,b,c\in X^{(0)}$.\begin{enumerate}
\item If $\A_c{a,b} >12\delta$, then all geodesics from $a$ to $b$ contain $c$.
\item If $e,e'$ are two edges originating at $c$ and that are on geodesics from $c$ to $a$, then $\A_c(e, \bar e') \leq 6\delta$. 
\item For any $\theta\geq 0$, if $\A_c(x,y) > \theta + 12\delta$ then for all choice of edges $\epsilon_1, \epsilon_2$ at $c$ starting geodesics respectively to $x$ and to $y$, $\A_c(\epsilon_1, \bar \epsilon_2) \geq \theta$.\end{enumerate}
\end{prop}
\begin{proof} (1) Applying $\delta$-thiness of triangles on the
 geodesics $[c,a ]$ and $[c, b]$ at distance $2\delta$ from $c$, we
 obtain a path of length at most $12\delta$, containing an arc of $[a,b]$, and that goes from the end of the first edge of
 $[c,a]$ to the end of the first edge of the $[c,b]$. If the angle at $c$ is larger than
 $12\delta$ this path cannot avoid $c$. But only the arc on $[a,b]$ can meet $c$ again. Hence the first claim.
 
 (2) Any geodesic bigon between $c$ and $a$ must be $\delta$-thin, so if we take a geodesic bigon starting with the edges $e$ and $e'$ and look at the points at distance $2\delta$ from $c$, we get a path of length $6\delta$ at most and that avoids $c$.
 
 (3) This is a 
 consequence of (1) and (2), combined with the triangular inequality for angles.
\end{proof} 
\begin{defi}A graph is called {\it fine} if for each edge $e$, and for all $L\geq 0$, the set of edges 
$$\{ e'\,|\,t(e')=o(e), \A_{o(e)}(e, e') \leq L\}$$ 
is finite. 
\end{defi}
This definition of fine graphs is equivalent to
Bowditch's definition that, for all edge, for all number, there are
only finitely many simple loops of this length passing through this
edge, \cite{BowRH}.

\begin{defi} Let $G$ be a finitely generated group, and  $H_1, \dots, H_k$,
subgroups of $G$.  Consider a Cayley graph ${\Cay} G$ of $G$ (over
a finite generating set), and construct the cone-off over $H_1, \dots,
H_k$ as follows: for all $i$ and all  left coset $gH_i$ of $H_i$,
add a vertex $\widehat{gH_i}$  and for each $h\in H_i$, add  an edge between 
$\widehat{gH_i}$ and $gh$.  We denote by $\widehat{{\Cay}} (G)$ this graph, called a 
{\it coned-off Cayley graph} (of $G$ with respect to the subgroups $H_1, \dots, H_k$).\end{defi}
We can now recall the definition of Bowditch of relative hyperbolicity from \cite{BowRH} (his original definition is about $G$-graphs, but this is an equivalent formulation from \cite[Appendix]{Dthesis}). 
\begin{defi}
A pair $(G,\{H_1, \dots, H_k\})$ of a group $G$ with a collection of
subgroups, is {\it relatively hyperbolic},  if a (equivalently any) coned-off
Cayley graph $\widehat{{\Cay}} (G)$ over $H_1, \dots,H_k$ is hyperbolic and fine. The groups $H_i$ are called {\it peripheral subgroups}.
\end{defi}
A difficulty of working in relatively hyperbolic
groups is the lack of local finiteness in the coned-off graph. The angles, and below the cones, are useful tools
for working around this. The following estimate says that the word metric is equivalent to the coned-off metric to which we add the sum of the angles at infinite valence vertices. We will be needing the first inequality only.
\begin{prop}\label{prop;pseudo_word_metric}
 Let $G$ be a  group hyperbolic relative to $\{H_1, \dots, H_k\}$ with $d_w$ a word metric, for a finite generating set containing
 generating sets of each $H_i$. Let $\widehat{{\Cay}} (G)$ be its coned-off Cayley graph,
 for its relatively hyperbolic structure, with $\hat{d}$ the graph metric.

 There are $A,B>1$ for which, for all $g\in G$,  for each choice of
 geodesic $[1,g]$ in the coned-off graph $\widehat{{\Cay}} (G)$, if $\sum\A ([1,g])$ denotes the sum of
 angles of edges of $[1,g]$ at vertices of infinite valence, then 
$$  \frac1A d_w(1,g) -B \leq     \hat{d}(1,g) +    \sum\A ([1,g])     \leq A\,d_w(1,g)  +B. $$
\end{prop}
\begin{proof}
First given a geodesic $[1,g]$ in the coned-off graph $\widehat{{\Cay}} (G)$, we construct a path in the
Cayley graph (whose length will bounded from above $d_w(1,g)$) by
replacing each pair of consecutive edges $e,e'$ at an infinite valence
vertex, by a path in the corresponding coset $xH_i$ of $H_i$.  We remark that
we can choose this path of length bounded above in terms of the
angle, say $\theta$. Indeed, by definion of angle, there exists a path in $\widehat{{\Cay}} (G)$ of length at
most $\theta$ avoiding $\widehat{xH_i}$, from $o(e)$ to
$t(e')$. Project each vertex in this path on the $1$-neighborhood of
$\widehat{xH_i}$. One gets a sequence of $\theta$ points $y_1, \dots
y_\theta$  in $xH_i$, from $o(e)$ to
$t(e')$,  two consecutive points being at an angle of at most $2\delta$ at
$\widehat{xH_i}$. It follows that there are only boundedly many
possible transitions $y_i^{-1}y_{i+1}$ in $H_i$, and we have our
bound on the word length in $H_i$ in term of the angle $\theta$ at
$\widehat{xH_i}$. Thus $  \frac1A d_w(1,g) -B \leq \hat{d}(1,g) +
\sum\A ([1,g]) $ (for some $A$ and $B$). 

For the second inequality, consider a geodesic from $1$ to $g$ in the word metric. It produces an injective continuous path $q$ in $\widehat{{\Cay}} (G)$, that fellow-travels a
geodesic $[1,g]$ in the coned-off graph (see for instance \cite[Proposition 8.25]{DS}). The path $q$ does not contain any vertex of infinite valence, therefore each time $[1,g]$ contains a vertex of
infinite valence, we can use the proximity of $[1,g]$ with $q$ and the
fact that $q$ does not contain this point, to estimate the angles in
terms of the length of a subsegment of $q$. By thinness of the graph, each subsegment of $q$ is only used at most a uniformly
bounded amount of times. Therefore, the total sum of the angles at infinite valence points is at most a certain multiple of the length of
$q$, which is the word length of $g$.
\end{proof}
\subsection{Cones}
The notion of angles will now allow us to talk about cones in any graph $X$. Cones are a useful tool for working in relatively hyperbolic groups as they allow to take neighborhoods of geodesics in a finite way. 
\begin{defi}
Let $X$ be a graph. Given an oriented edge $e$ and a number $\theta>0$, the
{\it cone of parameter $\theta$ around the oriented edge $e$}, denoted $\Co_{\theta}(e)$ is the
subset of vertices $v$ and edges $\epsilon$ of $X$ such that there is
a path from $o(e)$, starting by $e$,  and containing $v$ or $\epsilon$,  that is of length at most $\theta$ and for
which two consecutive edges make an angle at most $\theta$.

For vertices of finite valence we can define the cone at a vertex of parameter $\theta$ by the union of all the cones of
parameter $\theta$ around the edges adjacent to the vertex.
\end{defi}
If angles inside a cone of parameter $\theta$ are by definition bounded by $\theta$, angles at vertices close to the cone are also controlled in terms of $\theta$. Precisely we have the following estimate.
\begin{lemma}\label{lem;thetasquare}
In any graph $X$, for any oriented edge $e$ and $\theta\geq 0$, if $v\in \Co_\theta(e)$ then, for all $c\in X^{(0)}$ at distance $\leq
\theta/10$ from both $v$ and $e$, then $\A_c(o(e), v)\leq (\theta^2+3\theta)/2$.
\end{lemma}
\begin{proof}
Consider $c$ with two edges $\epsilon_1,
\epsilon_2$ starting at $c$, on geodesics $g_1, g_2$  toward $o(e)$, and $v$
respectively, of length $\leq \theta/10$. 
  We also are given a path $p$ starting by $e$, going to $v$
of length at most $\theta$, and maximal angle at most $\theta$ between
consecutive edges. Concatenating the paths $g_1$, $p$, $\bar g_2$
(orientation reversed), we
have a loop at $c$, starting by $\epsilon_1$, and ending by
$\epsilon_2$ and of length at most $6\theta/5$. If this loop doesn't pass by
$c$ (except starting and ending point), then by definition of angle we have that $\A_c(
\epsilon_1, \bar \epsilon_2 ) \leq  6\theta/5$, which would prove the claim. If this path passes at
least twice at $c$, only the segment $p$ can use $c$, and it can do so at most $(\theta- d(o(e),c)-  d(c,v))/{2}$ times, each time realizing an
angle between incident and exiting edge of at most $\theta$.  By the
triangular inequality on angles, this path gives a bound on the angle
$\A_c(\epsilon_1, \bar \epsilon_2)$ of at most $\theta(\theta- d(o(e),c)-  d(c,v))/{2}+   6\theta/5\leq (\theta^2+3\theta)/2$. This establishes
the claim.
\end{proof}
Cones also can be composed in an obvious way.
\begin{prop}\label{prop;composition_of_cones}
Let $\alpha,\beta\geq 0$, if $e,e', e''$ are three edges of a graph and if $e''\in
   \Co_\alpha(e')$ and $e'\in \Co_\beta(e)$ then $e''\in \Co_{\alpha+\beta}(e)$.
\end{prop}
\begin{proof}
We have a path of length and maximal angle bounded by $\beta$ that
starts by $e$ and ends by $e'$ or $\bar e'$, and another of length and
maximal angle at most $\beta$ that starts at $e'$ and contains $e''$
or $\bar e''$. Concatenate the two paths, possibly overlapping at
$e'$, or possibly cancelling a backtrack $e'\bar e'$,
produces   a  path of length at most $\alpha +\beta$ and maximal angle at
most $\max \{\alpha , \beta\}$. 
\end{proof}
Finally, the following proposition shows that cones are well-suited to work in hyperbolic and fine graphs.
\begin{prop}\cite{Dthesis},  \cite[\S2.3] {DYaman} \label{prop;cones_are_nicer}
In any fine graph, cones are finite.

In any $\delta$-hyperbolic graph, geodesic triangles are conically thin: for any geodesic triangle $([a,b], [b,c], [a,c])$, any edge $e$ on  $[a,b]$
 is contained in a cone of
parameter $50\delta$   around an  edge $e'$ that is either in $[a,c]$ at same distance from $a$ than $e$, or in $[b,c]$ at same distance from $b$ than $e$. 
\end{prop}

The statement above differs slightly from that in  \cite[\S2.3] {DYaman}, however the proof is the same. Let us present it, as it is very simple. Consider a subsegment $\sigma_1$ of $[a,b]$  that contains $e$, and extends to a distance $3\delta$ from both its end points (unless it reaches $a$ or $c$ before, a case we leave to the reader). Then append, at both ends, at most $2\delta$ long segments $\sigma_0$ and $\sigma_2$ in order to reach $[a,c] \cup  [b,c]$ (by usual thin triangles property).  If both $\sigma_0$ and $\sigma_2$ reach  $[a,c]$, or if both reach $ [b,c]$, then take the arc $\sigma_3$ of that segment, that closes a loop $\sigma_0\sigma_1\sigma_2\sigma_3$. Notice that $\sigma_3$ is of length at most $11\delta$ (and set $\sigma_4, \sigma_5$ to be trivial in order to continue syntactically thee argument). If the reached segments are different, close the loop as an hexagon around the center of the triangle, by an arc $\sigma_3$  of length $10\delta$ on $[b,c]$ of the triangle, an arc $\sigma_4$ of at most $2\delta$ between $[a,c]$ and $  [b,c]$, and an arc $\sigma_5$ back to the endpoint of $\sigma_2$ on $[c,a]$, hence of length at most $21\delta$. We thus get a loop $(\sigma_0\sigma_1\sigma_2\sigma_3\sigma_4\sigma_5)$ containing $e$. Triangle inequality forbids $e$ to be in the transitions arcs $\sigma_0, \sigma_2, \sigma_4$. If it is either in $\sigma_1$ or $\sigma_3$ there is actually nothing to prove.  Thus, we can extract the simple loop containing    $e$. Again triangular inequality ensures some distance between  $\sigma_0, \sigma_2, \sigma_4$, and this simple loop must contain an edge $e'$ of  $\sigma_2$ or $\sigma_4$ at same distance from respectively $a$ or $b$ than $e$. The length of the loop being at most $42\delta$, its maximal angle is bounded by this quantity. This makes  $e$ belong to the cone of parameter $42\delta$ around $e'$, and the proposition is proved.

The above proposition allows us to deduce that intervals are finite in coned-off graphs, even if those graphs are not locally finite.
\begin{prop}\label{prop;geod_finite}
In any hyperbolic fine graph, between any pair of points $a,x$ the set of points
in a geodesic between $a$ and $x$ is finite.
\end{prop}
A more important use of cones was explored in \cite{D_Israel1}, and follows from the fact that in a hyperbolic graph, quasi-geodesics ``without detours''  are conically
close to geodesic (\cite[Prop. 1.11]{D_Israel1}). We will use a variant of it, exposed in the next section. 
\subsection{Conical $\alpha$-geodesics}
In order to define a flow in the next section, we need to be able to travel coarsely from a point toward
another point in a uniformly finite way.  In \cite{AL}, the authors use the concept of $\alpha$-geodesic, for some $\alpha\geq 0$. Precisely, in a metric space, given two points $x$ and $a$, we say that $t$ belongs to an $\alpha$-geodesic between $x$ and $a$
if 
$$d(x,t)+d(t,a)\leq d(x,a)+\alpha.$$
Note that being on a $0$-geodesic is being on a geodesic. However, when the space is not uniformly locally finite, the set of points in a $\alpha$-geodesic
between $x$ and $a$ can be infinite in general, if $\alpha \geq 1$, even if we assume the graph to be fine. For instance, on the cone-off graph of a relatively hyperbolic group, when a geodesic passes through a cone point, any point on the coset of that cone point will be on a $2$-geodesic between the endpoints of that geodesic. But the coset points that are close (in the coset metric) to the entrance or exit points of the geodesic in the coset, will make small angles with some edges of the geodesic, whereas the other ones will make a large angle with any edge of the geodesic. We will use a variant of the notion of $\alpha$-geodesics, that takes angles into account. Constants are
not really relevant, but they need to be fixed by convention.
\begin{defi} \label{defi;slice}
Let $X$ be a graph, and $x,a\in X^{(0)}$. For $\rho\geq 0$ we denote by $\mathcal{E}_{a,x} (\rho)$ the set of edges $e$ with $d(a,o(e)) = \rho$ that are contained in geodesics from $a$ to $x$, or from $x$ to $a$. We say that $t\in X^{(0)}$ {\it belongs to an $\alpha$-conical-geodesic} between $x$ and $a$ if $t$ is on an $\alpha$-geodesic, $d(a,t)\leq d(a,x)$, and
$$ t\,\in\,    \bigcap_{e\in \mathcal{E}_{a,x}(d(a,t)) }
\Co_{40\alpha} (e).$$
In other words, points an an $\alpha$-conical-geodesic are off a geodesic only by a small angle. We will denote $U_{\alpha}[a,x]$ the set of points belonging to an $\alpha$-conical-geodesic between $x$ and $a$.

For each $\rho\geq 0$, we denote by $\mathcal{S}_{a,x}(\rho)$ the {\it slice of
$U_{\alpha}[a,x]$ at distance $\rho$ from $a$}, that is the set
$$\mathcal{S}_{a,x}(\rho)=\{t\in U_{\alpha}[a,x]\,|\,d(t,a)=\rho\}$$
 of points in $U_{\alpha}[a,x]$ at distance
$\rho$ from $a$.\end{defi}
We will now see some useful properties, namely that these sets are finite, stable (analogue of \cite[3.3]{AL} and \cite[3.4]{AL}), non-empty and slices at large angles are reduced to a point.
\begin{prop}\label{prop;recap}Let $X$ be a $\delta$-hyperbolic fine graph, and let $a,x \in X$ 
\begin{enumerate}
\item{\rm (Finiteness)}\label{lem;slices_are_finite} For any $\alpha>0$ the set
   $U_\alpha[a,x]$ is finite.   Each slice is contained in a cone of
   parameter $40\alpha$.
\item{\rm (Stability)}\label{lem;stab} If $\alpha =2\delta$, and  if $u\in U_{4\delta}[a,x]$ and $v\in U_{2\delta}[a,u]$ and if
  $d(u,v)\geq 5\delta$ then 
  $v\in U_{4\delta}[a,x]$.
\item\label{lem;slices_are_nonempty} 
 For all $v$ on a geodesic between $a$ and $x$, the slice of
  $ U_{2\delta}[a,x]$ at distance $d(a,v)$ from $a$ contains $v$,
  hence is non empty. 
\item\label{lem;un_goulet_de_slice}
If $\A_c(a,x) >(1000\delta)^2$, and $\alpha =2\delta$, then $\mathcal{S}_{a,x} (d(a,c))= \{c\}$.
\end{enumerate}
\end{prop}
\begin{proof}
(1) By definition $U_{\alpha}[a,x]$ is contained in an intersection of cones, which are finite according to Proposition \ref{prop;cones_are_nicer}. Each slice, by definition, is contained in $U_{\alpha}[a,x]$, which is an intersection of a family of cones of parameter $40\alpha$.

(2) That $v$ satisfies $d(a,v) + d(v,x) \leq d(a,x)+4\delta$ is \cite[3.4]{AL}.
    Consider $y$ in a geodesic  $[a,u]$ such that $d(a,y) = d(a,v)$, and an edge $e$ of
    $[u,a]$ with origin $y$. Then $v\in \Co_{80\delta}(e)$.  

    We now apply conical thinness of triangles (Proposition
    \ref{prop;cones_are_nicer}),   to the triangle  $([a,x], [x,u], [u,a])$, for the edge  $e\in [a,u]$. Since $e$ is at least $5\delta$-far from $u$,  the Proposition ensures that there exists  an edge $e'$ of $[a,x]$ at same
    distance from $a$ than $e$, for which 
    $e\in \Co_{50\delta} (e')$. It follows, by Proposition
    \ref{prop;composition_of_cones},  that $v\in \Co_{130\delta}
    (e')$, which establishes the claim.

(3) Call $[a,x]_1$  the given geodesic containing $v$. The claimed statement  amounts to check that, if  $[a,x]_2$ is any geodesic between $a$ and $x$, and $e$ is an edge of $ [a,x]_2$ starting  at distance $d(a,v)$ from $a$, then   $v$ is in   $\Co_{80\delta}(e)$. This is an  application of  the conical thinness of geodesic triangles, Proposition \ref{prop;cones_are_nicer} to the degenerate triangle $([a,x]_1, [x,x], [x,a]_2)$ (where $[x,a]_2$ is $[a,x]_2$ with reverse orientation).

(4) Consider two edges $e,e'$ at $c$ (i.e $o(e)=o(e')=c$) on a geodesic
$[a,x]$, where $t(e) $ is closer to $a$ and $t(e')$ is closer to
$x$. By definition of slices, one has $ \mathcal{S}_{a,x} (d(a,c)) \subseteq \Co_{80\delta}(e) \cap \Co_{80\delta}(e') \cap \{t, d(c,t)\leq 20\delta \}$.  Note that it implies   $ \mathcal{S}_{a,x} (d(a,c)) \subseteq \Co_{80\delta+1}(\bar{e}) \cap \Co_{80\delta+1}(\bar{e'}) \cap \{t, d(c,t)\leq 20\delta\}$.  We will prove that this intersection is reduced to $\{c\}$.   

By Proposition \ref{prop;goulet}, $\A_c( e, e') >(999\delta)^2$. 
 We apply Lemma  \ref{lem;thetasquare} for $\theta=200\delta$. For all $v\in \Co_{200\delta}( \bar e)$ different of $c$, but at distance at most $20\delta$ from it, one has $\A_c(o(\bar e), v) \leq ((200\delta)^2 + 600\delta)/2$. If $v$ is in  $\Co_{200\delta}(\bar{e'})$ as well, then  $\A_c(o(\bar e), v) \leq ((200\delta)^2 + 600\delta)/2$. By triangular inequality of angles at $c$,  $\A_c(\bar e, e')\leq (200\delta)^2 + 600\delta$, and this contradicts the previous lower bound of $(999\delta)^2$.  Thus, the intersection of $\Co_{200\delta}(\bar e)$, $\Co_{200\delta}(\bar {e'})$ and of the ball of radius $20\delta$ around $c$ is reduced to $\{c\}$. 

Thus, only $c$ can be in the
slice $\mathcal{S}_{a,x} (d(a,c))$. Since the slice is non-empty, we have the result.  
\end{proof}
\section{The flow on the coned-off graph}
Let us recall that, given a set $Y$ and an additive semigroup $V$,  a {\it flow} is a
map 
$\mathscr{F}:  Y\times V \to Y$ such that for all $y$,
$\mathscr{F}(y,0)=y$ and 
$\mathscr{F}(y, t_1+t_2) = \mathscr{F}(\mathscr{F}(y, t_1), t_2)$.  In
this definition, the semi-group $V$ is usually $\mathbb{R}_+$, but we
will use $\mathbb{N}$ (by convention $0\in \mathbb{N}$), and in this case the flow is entirely determined by the {\it flow step}, which is the map $\mathscr{T}:  Y\to Y$ defined
by $\mathscr{T}  (y) = \mathscr{F}(y,1) $.

In the following, $X$ will be a $\delta$-hyperbolic fine graph (for
some integer $\delta >0$), and $\Proba(X)$ is the space of probability measures
supported on $X^{(0)}$.    
    We will
define a  flow for the semi-group
$\mathbb{N}$, on the set  $Y= \Proba(X) \times X^{(0)}$, which is
equivariant by the group of isometries of $X$.  We think
of it as a version of a geodesic flow.

When
$X$ is countable (which is the case in our context) 
elements of $\Proba(X)$ are just functions from $X^{(0)}$ to
$\mathbb{R}_+$ of sum $1$, but we find it convenient and enlightening
to see them as probability measures.
\subsection{The flow step}\label{sec;def_flowstep}
In this subsection we will define the flow step 
$$ T: \left\{ \begin{array}{ccc} \Proba(X) \times X^{(0)}  & \to
                & \Proba(X) \times X^{(0)}   \\ 
  (\eta, a) & \mapsto & (T_a(\eta), a) \end{array} \right.$$

First let us introduce some vocabulary. Given $a$ and $x$, we say
that the step $T_a$ at $x$ is 
\begin{itemize}
  \item {\it initial} if $d(a,x)>5\delta$,
and $d(a,x)$ is not a multiple of $5\delta$, 
\item {\it regular}  if  $d(a,x)>5\delta$
and $d(a,x)$ is a multiple of $5\delta$,
  \item {\it ending} if $a$ is of infinite valence and $1<d(a,x)\leq
  5\delta$, or if $a$ has finite valence, $d( a,x) \leq 5\delta$  and
  there exists $c$ such that  $\A_c(a,x)  > 1000\delta$,  
  \item {\it stationary} in  all other cases , {\it i.e} in any of the following three cases 
\begin{itemize} 
\item $x=a$, 
\item  $a$ is of infinite
    valence and $x$ is a neighbor  $a$, 
\item  $a$ is of finite valence, $d( a,x) \leq 5\delta$  and  for all
  $c$,  $\A_c(a,x) \leq1000\delta$.    
\end{itemize}
\end{itemize}
We  denote by
$\delta_x $ the Dirac mass at $x$ and define the flow step $T$ by defining $T_a(\delta_x)$, for $a, x$ two vertices of $X$ and then extend by linearity on probability measures:  $$T_a(
\sum_x \lambda_x \delta_x) = \sum_x \lambda_x T_a(\delta_x).$$

For the following definition, we set the parameter $\alpha$ of the definition of slices, to be $\alpha = 2\delta$. 
\begin{defi}\label{defi:flowstep}Let $X$ be a $\delta$-hyperbolic fine graph, and let $a,x\in X^{(0)}$. We set $r_{a,x}$ to be the largest integer  $r$
such that  $5\delta\,r < d(a,x)$.  
\begin{itemize}
   \item   If the step  $T_a$ at $x$ is either initial or regular, then 
     $T_a(\delta_x)$ is the 
     uniform probability measure supported by  the slice
     $\mathcal{S}_{a,x}=\mathcal{S}_{a,x}(5\delta\,r_{a,x})$ (as in Definition \ref{defi;slice}).
   \item If the step $T_a$ at $x$ is ending, and $a$ has infinite
     valence then $T_a(\delta_x)$ is the 
     uniform probability measure supported by  the slice
     $\mathcal{S}_{a,x} =  \mathcal{S}_{a,x}(1)$  at
     distance $1$ from $a$. If the step is ending and $a$ has finite
     valence, then   $T_a(\delta_x)$ is the Dirac mass supported on
     the unique $c$ minimizing $d(a,c)$ such that $\A_c(a,x)  > 900\delta$. 
   \item If the step $T_a$ at $x$ is stationary, then $T_a(\delta_x) =
     \delta_x$ (and $\mathcal{S}_{a,x} = \{x\}$).
\end{itemize}
\end{defi}
Let us check that this definition makes sense. The set $\mathcal{S}_{a,x}$ is well defined:
in the case of an ending step, any vertex $c$ such that $\A_c(a,x)
> 900\delta$ is on every geodesic between $a$ and $x$, thus there is a
unique one that is closest to $a$. Then the set $\mathcal{S}_{a,x}$ is never empty by Proposition \ref{prop;recap} part (\ref{lem;slices_are_nonempty}) and it is always finite, by part (\ref{lem;slices_are_finite}) of that same proposition. Observe that among iterates of the flow step applied at $(\delta_x,
a)$, only the first iteration can be initial, only one step among the iterates can be ending, and after the ending step, all steps are stationary.  

This defines uniquely the value of $T_a(\delta_x)$ for all $x$ and all
$a$ and so completely defines the flow step $T$, hence also the flow 
\begin{align*}\mathscr{F}:\Proba (X) \times X^{(0)}\times\mathbb{N}&\to\Proba (X) \times X^{(0)}\\
(\eta,a,k)&\mapsto (T^k_a(\eta),a)\end{align*}
\subsection{Stationarity and mask} For two points $x$ and $a$ in $X^{(0)}$, the flow step $T_a(\delta_x)$ is a probability measure whose support is closer to $a$ than $\delta_x$, so the flow step has pushed the point $x$ towards $a$. Iterating this flow step will eventually stall and give a measure that is close to $a$. 
\begin{prop}For $X$ a $\delta$-hyperbolic fine graph and all $a,x\in X^{(0)}$, the flow step from Definition \ref{defi:flowstep} is stationary in $k$:  if 
 $k>r_{a,x}+1$, one has $T_a^{r_{a,x} +1}  (\delta_x)=T_a^{k}  (\delta_x)$. \end{prop}
\begin{proof}
If $y, y'$ are in the support of $T_a(\delta_x)$, one has $d(a,y) = d(a,y') \leq d(a,x)$. In case of equality, then $T_a(x) = \delta_x$. Indeed, if  the flow step is initial or regular, then $d(a,y) = d(a,y') < d(a,x)$,
 since by definition, $5\delta\,r_{a,x} < d(a,x)$.   The ending step case, also induces a strict inequality, and the stationary step case
 induces the equality $T_a(x) = \delta_x$. Hence the flow eventually become stationary for $k$ large enough.
\end{proof}
\begin{defi}\label{flot}
Given any two vertices $a$ and $x$ of $X$ a $\delta$-hyperbolic fine graph, the {\it mask of $a$ for $x$} is the probability measure given by
$$ \mu_x(a)=T^{R(a,x)}_a (\delta_x) \,  =\,  \lim_{k\to\infty} T^k_a (\delta_x),$$
where $R(a,x)$ the minimal number of iterations of $T_a$ on $\delta_x$ so that the stationary step is reached. 
\end{defi}
Let us observe the following.
\begin{prop}\label{prop;theta0}Let $X$ be a $\delta$-hyperbolic and fine graph. For $a,x\in X^{(0)}$, then
\begin{enumerate}
\item If $a$ is of infinite valence, for all $x$, the measure $\mu_x(a)$ is supported by the $1$-neighborhood of $a$. Moreover, if $e$ is the first edge of a
geodesic segment $[a,x]$, the measure  $\mu_x(a)$ is supported in ${\Co}_{160\delta}(e)$.

\item If $a$ is of finite valence, for all $x$, the measure $\mu_x(a)$ is supported by ${\Co}_{\theta_0}(a)$ for $\theta_0=1160\delta$. \end{enumerate}
\end{prop}

The last estimate is very crude, and with some closer look, one is
likely to get a much more precise control.

\begin{proof}
(1) If $a$ has infinite valence, by definition of the ending step in this case the measure is in the slice at distance 1 from $a$. By Proposition \ref{prop;recap}(\ref{lem;stab}), the last slice used is in a $4\delta$-conical geodesic from $a$ to $x$. So, given $e$ which is first edge on a geodesic from $a$ to $x$, by the definition of slices, this last slice used is contained in ${\Co}_{160\delta}(e)$. 

(2) In the case of a finite valence vertex $a$,  the condition for the
step being stationary is that there are no angle larger than
$1000\delta$ between $a$ and a point in the support. Again,  by Proposition \ref{prop;recap}(\ref{lem;stab}),
the
support is contained in 
    a slice    is in a
    $4\delta$-conical geodesic from $a$ to $x$,   
hence in a cone $\Co_{160\delta} (e_0)$ centered at an
edge in some geodesic $[a,x]$, at distance at most $5\delta$ from
$a$. We moreover know, by definition of the ending step, that there is
no angle larger than $1000\delta$ between $a$ and $o(e_0)$. Thus, by
composition of cones, Proposition \ref{prop;composition_of_cones}, 
 the support is contained in  $\Co_{1160\delta} (e)$ for an initial edge
of a geodesic $[a, x]$.
\end{proof}
The definition of the flow step being in a purely metric way, as a consequence the masks are equivariant with respect to the isometry group. 
\begin{prop}\label{prop;mask_equivariance}Let $X$ be a $\delta$-hyperbolic fine graph.
The map 
$$X^{(0)}\times X^{(0)}\to\Proba (X),\ (x,a)\mapsto \mu_{x}(a)$$
that associates to a pair of points $(x,a)$ the mask of $a$ for $x$ is
equivariant with respect to the isometry group of $X$. More precisely,
if $\phi$ is an isometry of $X$, one has $\phi_*\mu_{x}(a) = \mu_{\phi (x)} (\phi(a))$.
\end{prop}
%
\subsection{The flow stays focused} In this subsection we will be recording the behavior of the flow step at different stages
\begin{lemma}\label{focusLemma}Let $X$ be a $\delta$-hyperbolic fine graph, and for $a\in X^{(0)}$ let $T_a$ denote the flow step as in Definition \ref{defi:flowstep}. Let $x,x'\in X^{(0)}$.
\begin{enumerate}
\item \label{lem;stayfocus_regular}
If $x$ is such that the $k$ first iterates of $T_a$ are initial or regular steps, then, for any two points $y,y'$ in the support of $T^k_a(\delta_x)$, one has
$d(y,y') \leq 5\delta$. Moreover the support of $T^k_a(\delta_x)$ is contained in a cone of parameter $160\delta$. 
\item \label{lem;focus_initial} If $x,x'$ are neighbors, and $5k\delta<d(a,x)\leq d(a,x') < 5(k+1)\delta$, then the union of the supports of $T_a(\delta_x)$ and
$T_a(\delta_{x'})$ has diameter at most $8\delta+1$.
\item \label{lem;elementary_confluence} If $d(x,x')\leq 8\delta$, with $d(x,a) = d(x',a)$ and assume that for both of them, the flow step is regular. Then 
the union of the supports of both measures  $T_a(\delta_x)$  and
$T_a(\delta_{x'})$ has diameter at most $8\delta$. Moreover, the union of these supports is contained in a cone of
parameter $170 \delta$, and their intersection is not empty. 

\item \label{prop;firststep} If $x,x'$ are neighbors and both at at distance larger than $5\delta$ from $a$, then there are two exponents $\epsilon,
\epsilon'$ that are either $1$ or $0$, such that
$T^\epsilon_a(\delta_x)$ and $T^{\epsilon'}_a(\delta_{x'})$ have support
on a sphere centered at $a$ of radius $5k\delta$ for some $k$,
and the diameter of their union is at most $8\delta+2$.
\end{enumerate}
\end{lemma}
\begin{proof}
(1) By Proposition \ref{prop;recap}(\ref{lem;stab})   $y$ and $y'$ are on $U_{4\delta}[a,x]$.
As in \cite{AL} we can estimate:  
$$d(a,x) + d(y,y')  \leq \max \{   d(a,y) + d(y',x),
\,   d(a,y') + d(y,x) \}  +\delta \; \leq d(a,x) + 4\delta  +\delta.$$
Hence $d(y,y') \leq 5\delta$. Then, observe that by Proposition \ref{prop;recap}(\ref{lem;stab}) the support of
$T^k_a(\delta_x)$ is in a slice of  $U_{4\delta} [a,x]$.
Proposition \ref{prop;recap}(\ref{lem;slices_are_finite}) concludes. 

\smallskip

(2) By assumption the steps of the flow are initial. Let  $y_0$ and $y'_0$ be respectively on geodesics $[x,a]$ and $[x',a]$ at distance $5k\delta$ from $a$. By $\delta$-hyperbolicity, they are at distance less than $2\delta+1$ from each other. Let $y,  y'$ respectively be in the supports of  $T_a(\delta_x)$ and $T_a(\delta_{x'})$.    Apply the argument of part (\ref{lem;stayfocus_regular}) using  that $y$ and $y'$ are on $2\delta$-geodesics (there is only one step of flow). One obtains that $d(y,y_0) \leq 3\delta$ and $d(y', y_0') \leq 3\delta$, hence $d(y,y')\leq 8\delta+1$.

\smallskip

(3) Here $x$ and $x'$ are not assumed neighbors, but at distance $5(k+1)\delta$ from $a$. Let $y_0$ and $y'_0$ respectively be on geodesics $[x,a]$ and $[x',a]$ at distance $5k\delta$ from $a$. By $\delta$-hyperbolicity, they are at distance less than $2\delta$ from each other and the same argument than part (\ref{lem;focus_initial}) gives a bound of $8\delta$ on the union of the supports of $T_a(\delta_x)$ and $T_a(\delta_{x'})$


Given two geodesics $[a,x]$ and $[a,x']$, we denote by $e, e'$ the edges of these geodesics starting at distance $\left( d(a,x) -5\delta\right)$ from $a$. By $\delta$-hyperbolicity, the two segments stay $2\delta$-close for a length of at least $\delta$ after $e$ and $e'$. One deduces (as in the proof of
Proposition \ref{prop;goulet}) that $e\in \Co_{10\delta}(e')$. Therefore, by composition of cones (Proposition \ref{prop;composition_of_cones}) $\Co_{160\delta}(e) \subseteq \Co_{170\delta} (e')$. Moreover, $o(e)$ is at distance at most $\delta$ from $[a,x']$, so is on a $2\delta$-geodesic between $a$ and $x'$. Also,   as we discussed, it is  in   $\Co_{10\delta}(e')$, for any edge $e'$ on any geodesic  $[a,x']$ satisfying that $e'$ starts at  $\left( d(a,x) -5\delta\right)$ from $a$.  Thus,    $o(e) \in U_{2\delta}[a,x']$. and hence by Proposition \ref{prop;recap}(\ref{lem;slices_are_nonempty}) the
intersection of the supports is non-empty.

\smallskip

(4) There are several cases, depending whether the flow from $x$
respectively from $x'$ needs an initial step or not, in other words,
whether $x$, respectively $x'$ are at distance strictly greater than $5k\delta$
or equal $5k\delta$ from $a$.

If the first step of the flow is initial for both $(\delta_x, a)$ and
$(\delta_{x'}, a)$, then part (\ref{lem;focus_initial}) says that the union of
the supports of $T_a(\delta_x)$ and $T_a(\delta_x)$  has diameter at
most $8\delta$. The same is true according to part (\ref{lem;elementary_confluence}) if both first steps of the flow are
regular. 

Assume that $d(a,x) =  5k\delta$ and $d(a,x') >5k\delta$. Then we compare $\delta_x$ (i.e $\epsilon =0$)  with
$T_a(\delta_{x'})$ (i.e $\epsilon' =1$). According to part (\ref{lem;stayfocus_regular}) the support of  $T_a(\delta_{x'})$ has diameter
at most $5\delta$, and by Proposition \ref{prop;recap}(\ref{lem;slices_are_nonempty}) it contains a point at
distance $1$ from $x'$, hence at distance less than $2$ from $x$, the union  of
the supports of $T^\epsilon_a(\delta_x)$ and
$T^{\epsilon'}_a(\delta_{x'})$ has diameter at most $5\delta+2$, hence
the result.
\end{proof}
\subsection{Check points at large angles}
The following proposition says that when the flow passes a large
angle, this large angle acts as a check point: The flow could have started at this large angle, and still give the
same result.
\begin{prop} \label{prop;checkpoint} 
Let $(X,d)$ a $\delta$-hyperbolic fine graph. Let $a, x$ be vertices in $X$.  If there exists $c\in [a,x]$ such that $\A_c(a,x)>(2000 \delta)^2$, then for any
$x'$ at distance $1$ from $x$, we have that 
$$\mu_x(a) = \mu_{x'}(a)= \mu_c(a).$$
\end{prop}
\begin{proof}Define $r(c)$ to be the largest $r$ so that  $(r_{a,x}  -r )5\delta\geq d(a,c)$. First notice that for all $r\leq r(c)$, and  all $b$ in the support of $ T^r_a(\delta_x) $, one has 
$$\A_c(a,b)>(1910\delta)^2.$$ 
Indeed, let $b$ in the support of $ T^r_a(\delta_x) $. Then $d(a,b) =  (r_{a,x}  -r )5\delta  >d(a,c)$, and by Proposition \ref{prop;recap}(\ref{lem;slices_are_nonempty}) there is $b'$
on the geodesic $[a,x]$ realizing the maximal angle at $c$, that is
  in the support of $ T^r_a(x) $. Hence   $\A_c(a,b')
  >(2000\delta)^2$. Since $b'$ and $b$ are in a same cone of parameter
  $160\delta$ (Proposition \ref{prop;recap}(\ref{lem;slices_are_finite})),  the maximal angle between then is at most
  $2(160\delta)^2$ (by Lemma \ref{lem;thetasquare}). Thus,
  $\A_c(a,b)>(1987\delta)^2>(1910\delta)^2$ as claimed.

We now show that $T^{r(c)+1}_a(\delta_x)=T_a(\delta_c)$. If $d(a,c)$ is a multiple of $5\delta$, for any point $b$ in the support of $ T^{r(c)-1}_a(\delta_x)$ then $\A_c(a,b)>(1910\delta)^2$. Thus, all geodesics from $a$ to $b$
  pass through $c$ and make an angle at least $(1900\delta)^2$. 
      It follows that the intersection of all cones of
  parameter $80\delta$ centered at an edge of
  $\mathcal{E}_{a,b}(d(a,c))$ is reduced to $\{c\}$ (by Proposition \ref{prop;recap}(\ref{lem;un_goulet_de_slice})), hence $T_a(\delta_b) = \delta_c$, which means that
  $T^{r(c)}_a(\delta_x) = \delta_c$, and $T^{r(c)+1}_a(\delta_x)=T_a(\delta_c)$ as claimed. When  $d(a,c)$ is not a multiple of $5\delta$,  apply $T_a$ to all points $b$ in the support
  of  $ T^{r(c)}_a(\delta_x) $. But since we saw that $\A_c(a,b)>(1910\delta)^2$, all geodesics from $a$ to $b$
  pass through $c$ and make an angle at least $(1900\delta)^2$. The
  points on
  $2\delta$-geodesics from $a$ to $c$ are exactly the points on
  $2\delta$-geodesics from $a$ to $b$ that are at distance at most
  $d(a,c)$ from $a$, and we have equality of the sets $\mathcal{E}_{a,b}((r_{a,x}- r(c))5\delta)=\mathcal{E}_{a,c}((r_{a,x}- r(c))5\delta)$. Hence  $T_a(\delta_b) = T_a(\delta_c)$, and and $T^{r(c)+1}_a(\delta_x)=T_a(\delta_c)$ again as claimed. 
   
 Finally, if $x'\neq c$ and $x'$ at distance $1$ from $x$, then, for $r'(c)$ defined similarily to $r(c)$ for $x'$ we have 
$$T^{r(c)+1}_a(\delta_x) = T^{r'(c)+1}_a(\delta_{x'}) = T_a(\delta_c).$$
The angle at $c$ between $a$ and $x'$ can be reduced by one unit, which is still well above the threshold to apply the previous argument. Iterating the flow step gives the conclusion of the proposition.
\end{proof}
\subsection{Confluence}
We will need to control that, given $a$, and $x,x'$ close to each
other, but ``far'' from $a$ (in terms of distance or of angles), each of the regular flow steps from $(\delta_x, a)$
and from $(\delta_{x'}, a)$,  are uniformly close. To do that, we will be talking of confluence. Recall that the norm $\|\eta\|$  of a measure is defined as its $\ell^1$-norm $\|\eta\| = \sum_{x\in X} |\eta(x)|$. For $\eta, \eta'$ two probability measures on $X$, denote by
$\mathscr{M}(\eta,\eta')$ the measure defined by 
$$\mathscr{M}(\eta,\eta') (\{y\}) = \min \{ \eta (\{y\}), \eta'(\{y\}) \}.$$
We also define the {\it symetric difference} of $\eta$ and $\eta'$ as 
$$(\eta  \Delta\eta' ) = \eta+\eta' - 2\mathscr{M}(\eta,\eta'),$$
which is always a positive measure since $\eta, \eta'$ are
probability measures. Moreover,  $\|\eta  \Delta\eta'\| = \| \eta - \eta'\|$.
 \begin{defi}\label{defconf}
Let $\beta \geq 0$. We say that the flow step $T_a$ is {\it $\beta$-confluent} on $(\eta,
\eta')$ if 
$$\|   \mathscr{M}  (T_a (\eta),  T_a(\eta'))     \| \geq
    \| \mathscr{M}  (\eta, \eta') \| +   \beta  \|  \eta \Delta
    \eta'\|.$$  
%
\end{defi} 
\begin{rem}\label{betterconf}Notice that the flow step $T_a$ is $\beta$-confluent on $(\eta,\eta')$ if and only if   
$$\|  T_a(\eta)   \Delta  T_a(\eta')   \|   \leq (1- 2\beta)  \|  \eta   \Delta
 \eta'  \|.$$
Indeed, observe that $\|  T_a(\eta)   \Delta
 T_a(\eta')   \|   = 2- 2 \|  \mathscr{M}  (T_a (\eta),  T_a(\eta'))
 \|$, and similarily  $\| \eta   \Delta
 \eta'   \|   = 2 - 2 \mathscr{M}  (\eta, \eta')$. After substitution, and dividing by $2$, the above condition reads: 
$$     1 -  \|  \mathscr{M}  (T_a (\eta),  T_a(\eta'))
 \|    \leq    ( 1 -  \|\mathscr{M}  (\eta, \eta')  \|)     - \beta   \|  \eta   \Delta
 \eta'  \|    $$   which is equivalent to the inequality of the
        definition. 
\end{rem}
\begin{prop} \label{prop;decay_in_C} Let $X$ be a fine $\delta$-hyperbolic graph with a co-finite group action. Let $k\geq 1$, and $\eta, \eta'$ be two measures whose supports are on $S(a,5(k+1)\delta)$, and have union of diameter less than $8\delta$.

Then, there is a constant $C>1$ such that the $k$ consecutive flow steps on $(\eta, a)$ and $(\eta',a)$ are $\frac{1}{C}$-confluent. In particular, for all $m\geq k$, 
 $$   \|  T^m_a(\eta)   \Delta
 T^m_a(\eta')   \|   \leq  \left( 1-\frac{2}{C}\right)^k  \|  \eta   \Delta
 \eta'  \|.   $$
\end{prop}
\begin{proof}
First notice that the assumption that $X$ is a fine $\delta$-hyperbolic graph with a
co-finite group action implies that here are only finitely many orbits of cones
of given parameter, and in each of these orbits, cones are finite (Proposition \ref{prop;cones_are_nicer}), with same cardinality. Hence,  there exists $C>1$ that is an upper bound on the cardinality of each cone of parameter
$160\delta$ in $X$.

We first proceed by induction  to prove that for each
$i\leq k$,   the union of supports of the measures $T^i_a(\eta)$ and
$T^i_a(\eta')$ have diameter at most $8\delta$. There is nothing to
prove for $i=0$. For $i>0$,  assuming  $T^{i-1}_a(\eta)$ and
$T^{i-1}_a(\eta')$ have support  whose union has diameter at most
$8\delta$, we can say that, since all  steps of the flow so far were
regular, these supports are on the sphere $S(a, 5
(k+2-i)\delta)$. We apply Lemma \ref{focusLemma} part (\ref{lem;elementary_confluence}) to obtain
that the union of supports of $T^{i}_a(\eta)$ and
$T^{i}_a(\eta')$ have diameter at most $8\delta$. 

We then prove the confluence. For each Dirac masses $\lambda_v\delta_v$
and $\lambda_{v'}\delta_{v'}$, respectively in the support of
$T^{i-1}_a(\eta) - \mathscr{M}  (T^{i-1}_a (\eta),  T^{i-1}_a(\eta'))  $ and of  $T^{i-1}_a(\eta')-  \mathscr{M}  (T^{i-1}_a (\eta),  T^{i-1}_a(\eta'))  $,
the regular step of the flow sends each of these Dirac masses on a uniform measure on a subset of a  slice, in a cone of parameter $160\delta$, by Proposition \ref{prop;recap} part (\ref{lem;slices_are_finite}). Moreover the respective supports of these uniform measures share at least one point, by  Lemma  \ref{focusLemma} part (\ref{lem;elementary_confluence}). Therefore, this step is 
$\frac{1}{C}$-confluent on the Dirac masses $\lambda_v\delta_v$ and $\lambda_{v'}\delta_{v'}$ (see Definition \ref{defconf} above), where $C$ is the upper bound on the cardinality of cones of parameter $160\delta$ in $X$. Precisely, according to Remark\ref{betterconf} we have that
$$\| T_a(\delta_v)   \Delta  T_a(\delta_{v'})   \|   \leq (1- \frac{2}{C})  \|\delta_v   \Delta\delta_{v'} \|.$$
Since the flow step on measures is defined by linearity on Dirac masses, this confluence on
all masses in these supports gives that 
$$ \|T^{i}_a(\eta) \Delta T^{i}_a(\eta')\| = \|T(T^{i-1}_a(\eta)) \Delta T(T^{i-1}_a(\eta'))\| \leq (1-\frac{2}{C})  \|   T^{i-1}_a(\eta)
\Delta T^{i-1}_a(\eta')   \| $$
\end{proof}
\begin{prop}\label{prop;kappa} Let $X$ be a hyperbolic fine graph, with a cofinite group
action. There exists a constant $\kappa  <1$ such that, for all $a$,  for all
$x_1, x_2$ at distance $1$ from each other, if $x$ is among $x_1,x_2$, closest to $a$, and if $\Theta$ denotes
the minimal sum of  angles at infinite valence vertices on a geodesic
$[a,x]$, then for $d(a,x)+\Theta $ sufficiently large,  one has
$$ \| \mu_{x_1}(a) \Delta  \mu_{x_2}(a) \| \, \leq \,
\kappa^{d(a,x)+\Theta} $$

\end{prop}

Observe that the statement would hold also for $\Theta$ defined as the
sum of all consecutive angles, and also if it was taken to be the
maximum of these sums.

\begin{proof}
If one of the angles on $[a,x]$  is larger than $(2000\delta)^2$ then by
Proposition \ref{prop;checkpoint}, $\mu_{x_1}(a) \Delta  \mu_{x_2}(a)
=0$. If all the angles on $[a,x]$ are smaller that $(2000\delta)^2$,
then $$  d(a,x)   \leq  d(a,x)+\Theta     \leq   (2000\delta)^2 d(a,x).$$

It follows that $(\kappa)^{d(a,x)+\Theta}  \geq
(\kappa^{(2000\delta)^2 6\delta } )^{d(a,x)/6\delta}$. We then
choose $\kappa<1$ sufficiently  close to $1$ so that $\kappa^{(2000\delta)^2 6\delta } >
(1-\frac{2}{C})$ from Proposition \ref{prop;decay_in_C}. If $d(a,x)$ is
sufficiently large (uniformly), $d(a,x)/6\delta$ is larger than the
number of regular step of the flow applied to $(\delta_x, a)$. 
 Thus, Proposition \ref{prop;decay_in_C}, and Lemma \ref{focusLemma} part (\ref{prop;firststep}) give the conclusion.
\end{proof}
\subsection{Non-confluence} In this subsection we will be showing that for two given sources of the flow, any point on a geodesic between those two sources have disjoint masks for each of those sources.
\begin{prop}\label{non-conf}Let $X$ be a $\delta$-hyperbolic and fine graph, and let $\mu$ be the mask of Definition \ref{flot}.
\begin{enumerate} 
\item If $d(x,x')\geq 10\delta$ and if $a$ is of finite valence,  at
distance at least $5\delta$ from both $x$ and $x'$,  and $a\in [x,x']$ for some geodesic, then
$\| \mu_x(a) \Delta \mu_{x'}(a) \| =2$.
\item If  $a$ has infinite valence, and $\A_a(x,x') > 10(160\delta)^2 $
then $\| \mu_x(a) \Delta \mu_{x'}(a) \| =2$.
\end{enumerate}
\end{prop}
\begin{proof}
(1) We proceed by contradiction. 
Assume that $v$ is in the support of both $\mu_x(a)$  and $\mu_{x'}(a)
$.

In the first case, assume that no step in the flow was {\it
  ending}. Then $v$ is at distance $5\delta$ from $a$ and is in
$4\delta$-geodesics between both $a,x$ and $a,x'$ (by Proposition \ref{prop;recap}(\ref{lem;stab})). This is not
possible. 

Assume that one step of the flow for $\delta_x$ was ending. If the there was
no step of the flow of type ending for $\delta_{x'}$, then the
supports of  $\mu_x(a)$  and $\mu_{x'}(a)$ are on different spheres
around $a$, so they are disjoint.  If there was an ending step for
$\delta_{x'}$, then, since $a$ is of finite valence,  $\mu_x(a)$  and
$\mu_{x'}(a)$  are dirac masses carried by vertices $y,y'$ on which,
respectively, $\A_y(a,x)$ and $\A_y(a,x')$ are larger than
$20\delta$. Since $a$ is on a geodesic from $x$ to $x'$, by
Proposition \ref{prop;goulet} one has
$y\neq y'$.

(2) The assumption indicates that the flow from $\delta_x$ to $a$ and from
$\delta_{x'}$ to $a$ has an ending step. It follows from Proposition
\ref{prop;theta0}  that the measures
$\mu_x(a)$  and $ \mu_{x'}(a)$ have supports  contained in
two cones of parameter $160\delta$ centered respectively on an edge starting
a
geodesic $[a,x]$ and starting  a geodesic $[a,x']$. Let $e,e'$ these
edges. One has $\A_a(e,e') > 9 (160\delta)^2$. It follows from Lemma \ref{lem;thetasquare}
that the intersection of the cones is reduced to $\{a\}$. Thus the
intersection of the supports is empty.
\end{proof}
\begin{cor}\label{cor;ALconique}
 Let $X$ be the coned-off graph of a finitely generated relatively
   hyperbolic  group $G$
   over its family of peripheral subgroups. There exists $p>1$, and $\varepsilon >0$  such that for all $x$    
 $$\sum_{a\in X} \| \mu_1(a) \Delta \mu_{x}(a) \|^p $$ 
 is convergent,
 and if $x$ is  at distance at least $10\delta$ from $1$, the sum   is larger than $\varepsilon d(1,x)$.  
\end{cor}
\begin{proof}
We identify $G$ with the image of its Cayley graph in $X$.
For the first statement, let $\Theta(x,x')$ denote
the minimal sum of  angles at infinite valence vertices on a geodesic
$[x,x']$.  Define $d': X\times X \to \mathbb{N}$ as
$d'(g_1,g_2) = d_X(g_1,g_2) + \Theta(x,x')$. 
By Proposition \ref{prop;pseudo_word_metric}, $d'$ is coarsely larger
than a word metric $d_w$ on $G$.   In particular,
there is $A\geq 1$ and $ \gamma_G >1 $ such that $S_{d'}(R,G) = \{ g\in G, d'(1,g)
 =  R\}$ has cardinality less than
$A\gamma_G^R$.  Denote by  $S_{d'}(R)$ the union of
$S_{d'}(R,G) $ and of the vertices $x$ of $X$ of infinite valence such
that $d'(1,x)= R$.  One can check that its cardinality is at
most  $A' \gamma_G^R$ for a certain constant $A'$ (related to
the maximal order of intersection of two peripheral subgroups).

Recall that we defined a constant $\kappa$ in Proposition \ref{prop;kappa}.  We choose $p$ such that $\kappa^p< 1/\gamma_G$. Then,
fixing $x$,  we
have the bound
$$ \sum_{a\in S_{d'}(R)} \|   \mu_1(a) \Delta \mu_{x}(a)     \|  \leq
\kappa^{p R}  A' \gamma_G^{ R}   $$
This defines a summable family of numbers therefore $\sum_{a\in X} \| \mu_1(a) \Delta \mu_{x}(a) \|^p $ is convergent.

The lower bound is given by Proposition \ref{non-conf} and the fact that there are at least
$d(x,x')/2 - 1$ vertices of finite valence between $x$ and $x'$.
\end{proof}



\section{An action that is proper in the coned-off distance}\label{proper_in_CO}
In this section $X$ is the coned-off Cayley graph of a relatively
hyperbolic group $G$, and $\mathscr{V}=\mathbb{R}X^{(0)}$ is the vector space of all functions
from $X^{(0)}$ to $\mathbb{R}$ with finite support, that we endow with the $\ell^1$-norm. For all $p >1$ or $p=\infty$, we consider $\mathscr{W}_p= \ell^p(X^{(0)},
\mathscr{V})$.  In other words, an element $\omega$ of $\mathscr{W}_p$ is the data, for all $a\in
X^{(0)}$, of a vector $\omega_a\in \mathscr{V}$,
such that its norm  
$$  \|  \omega \|_{\mathscr{W}_p} = \left( \sum_a \|  \omega_a \|^p_{\mathscr{V}} 
\right)^{\frac1p}  = \left( \sum_a \left( \sum_v |\omega_a(v)|
  \right)^p   \right)^{\frac1p}$$   
  is finite. Let $ {\mathscr{W}} $ be one of the  ${\mathscr{W}}_p$. 
We denote by $\vv{{\Isom}} {\mathscr{W}}$ the unitary group of
${\mathscr{W}}$, {\it i.e}  the subgroup of
linear automorphisms of ${\mathscr{W}} $ that are isometries for the norm. The group $G$ admits an isometric linear action on ${\mathscr{W}} $  by
precomposition by isometries of $X$.  Let us denote
$\pi: G\to \vv{{\Isom}} {\mathscr{W}} $ this
action. We want to promote it into an action by affine isometries that has no fixed point. For that we need a cocycle, as we recall
in the next subsection.
\subsection{Generalities on cocycles and actions} 
Recall that  we may see $ {\mathscr{W}} $ as an affine space, and also
as an abelian group (for the addition of vectors).  The
 group of affine isometries of (the affine space) ${\mathscr{W}} $  is the semidirect
product $$ {\Aff \Isom}  {\mathscr{W}}  =  {\mathscr{W}} \rtimes  \vv{{\Isom}} {\mathscr{W}} $$ for
the natural action of $\vv{{\Isom}} {\mathscr{W}} $ on
${\mathscr{W}} $. Given a group $G$, any homomorphism  $\phi:G\to {\Aff \Isom}{\mathscr{W}}$ has an image in
$\vec{\phi}:G\to \vv{{\Isom}} {\mathscr{W}} $ through the quotient map, and,
given this homomorphism $\vec{\phi}$, the homomorphism
$\phi$ is
characterised by a map $c:G\to {\mathscr{W}} $ recording $c(g) =
\phi(h)(\vec{0}_ {\mathscr{W}})$.  Applying the law of the
semidirect product reveals that $c$ satisfy the cocycle relation 
$$c(g_1g_2) = c(g_1) + \vec{\phi}(g_1) (c(g_2)),$$
for all $g_1,g_2\in G$.
The cocycle is actually the difference between $\phi$ and a given section of
$\vec{\phi}$ in the semi-direct product. Conversely, whenever one has $\vec{\phi}:G\to \vv{{\Isom}}
{\mathscr{W}}$  and $c:G\to  {\mathscr{W}}$ satisfying the cocycle
relation for $\vec{\phi}$, one can define the map 

$$\phi:G\to  {\Aff \Isom}  {\mathscr{W}}, g\mapsto \phi(g) = (c(g),  \vec{\phi}(g)),$$
which is a homomorphism.

\begin{defi}\label{proper}For a finitely generated group $G$, endowed
  with a (locally finite) word metric $d$, the
homomorphism $\phi$ is called {\it proper}  if $c$ is
proper, or equivalently  if $\|c(g_i)\|$
goes to infinity for any sequence $(g_i)$ of elements in $G$ for which
$d(1,g_i)$ goes to infinity in $\mathbb{R}$.

 If $d$ is a left invariant metric obtained from a coned-off cayley
 graph over the cosets of subgroups, (or, in other words,  a relative
 word metric), then $\phi$ is relatively proper if   $\|c(g_i)\|$
goes to infinity for any sequence $(g_i)$ of elements in $G$ for which
$d(1,g_i)$ goes to infinity. 
\end{defi}
\subsection{A relatively proper affine isometric action}
We return to our initial context. We have $G$ a relatively hyperbolic
group, and $\pi: G\to
\vv{{\Isom}} {\mathscr{W}} $ for our relatively hyperbolic group
$G$, which we want to promote to a homomorphism in ${\Aff}{\Isom} {\mathscr{W}} $. We thus need a cocycle for
$\pi$.
We identify $G$ with the vertices of its Cayley graph $\Cay G$ in the
coned-off Cayley graph $X$. We define the map $c$ as follows
$$c: \left\{ \begin{array}{ccc}  G &\to &
                                                           \mathscr{W}_{\infty}
                                \\  g & \mapsto &  \{a \mapsto
                                                 \mu_1(a) -
                                                  \mu_g(a)
                                                 \}. \end{array} \right.$$
                                                  %
where $\mu_x(a)$ is the probability measure from Definition \ref{flot}. Recall that  
$\mu_1(a) -     \mu_g(a)$ has same $\ell^1$-norm as   the symmetric difference 
$\mu_1(a) \Delta\mu_{g}(a)$ because the latter is the absolute value of the former.
The following proves the first part of our main result, Theorem \ref{main}
\begin{thm}\label{thm;actionAL}
 Let $G$ be a relatively hyperbolic group and  let $X$ be the
  coned-off Cayley graph of $G$ with respect to its peripheral
  subgroups. For $p>1$ large enough, $c$ has its values in $\mathscr{W}_p=\ell^p(X^{(0)},\mathscr{V})$, and is a cocycle for the representation
 $\pi$ on $\mathscr{W}_p$. 
 
 Moreover, the homomorphism $(c,\pi): G \to {\Aff}{\Isom}
 {\mathscr{W}}_p$ is relatively proper for the metric on $G$ induced by the
coned-off Cayley graph. Precisely, there is
  $\epsilon >0$ and a constant $k_0$,  for which  for all $g\in G$,  $\|g(\vec{0})\|^p
  \geq \epsilon d_X(1,g) -k_0$.  

\end{thm}
\begin{proof} That $c$  is a cocycle for $\pi$ is a standard computation using Proposition \ref{prop;mask_equivariance}. It takes its values in $\ell^p(X^{(0)},\mathscr{V})$, by  Corollary \ref{cor;ALconique}.    The estimate on $\|g(\vec{0})\|$ is a consequence of the last part of Corollary \ref{cor;ALconique}.\end{proof}
\section{Random coset representatives and subgroup
  properness} \label{induced} We will now recall some basic notions on
induced representations from a subgroup $H$ to a larger group $G$, and
discuss random coset representatives (see Definition
\ref{randomCoset}). Our main result is that in case those random coset
representatives that are $\ell^p$-almost $G$-invariant exist
(Definition \ref{lpalmost}), we can define a cocycle on $G$ for the
induced representation from $H$, that will have the same
$H$-properness than the cocycle we started with (Proposition \ref{prop;induced_affine_rep}). We start by recalling some basic notions on Banach spaces.
\begin{defi}\label{uc}A Banach space $B$ is called {\it uniformly convex} if for any $0 < \epsilon \leq 2$, there is $\delta(\epsilon) >0$ such that for any $x,y\in B$ with $\|x\|=\|y\|=1$ and $\|x-y\|\geq\epsilon$, then
$$\frac{\|x+y\|}{2}\leq 1-\delta(\epsilon).$$
Suppose that a finite number of Banach spaces, $B_1,B_2, \dots, B_k$ are given, and that $B$ is their product. We shall call $B$ a {\it uniformly convex product of the $B_i$} if the norm of an element $x=(x_1, x_2,\dots, x_k)$ of $B$ is defined by
$$\|x\|=N(\|x_1\|,\|x_2\|,\dots,\|x_k\|),$$
where $N$ is a continuous non-negative function, homogeneous, strictly convex and strictly increasing.
\end{defi}
A classical example of uniformly convex product is an $\ell^p$-combination. According to Clarkson \cite{Cla}, $\ell^p$-spaces are uniformly convex, and any uniformly convex product of uniformly convex Banach spaces is again uniformly convex. According to Day \cite{Day}, if $B$ is a uniformly convex Banach space, then the $B$-valued functions that are $\ell^p$ are uniformly convex as well. 
\subsection{Induced representations}

Assume that $H$ is a subgroup of a group $G$. Assume that we have a 
representation $\pi_H: H \to \vv{ {\Isom} } \mathscr{V} $ of $H$
in the linear isometry group of a normed vector space $\mathscr{V}
$ in other words, a unitary representation on $\mathscr{V}
$). Then, there is a normed vector space $\mathscr{V}' $, with an
isometric embedding of $\mathscr{V}$ in it,  and a
 unitary representation $\pi_G: G \to \vv{ {\Isom} } \mathscr{V}' $, called
the induced representation of $H$ to $G$, such
that the restriction to $H$ preserves  $\mathscr{V}$  and induces
$\pi_H$. We refer the reader to  \cite[\S E.1]{BdlHV} for the material recalled below.
First, let us recall the notion of induced representation.
\begin{defi}Given a representation $(\pi_H,V)$ of a group $H$ that happens to be a subgroup of a bigger group $G$, we define
$${\mathcal A}_\infty=\{\varphi:G\to V\,|\,\varphi(gh)=\pi_H(h^{-1})\varphi(g)\hbox{ for all }g\in G, h\in H\}.$$
We can endow the vector space ${\mathcal A}_\infty$ with the induced action from $\pi_H$
\begin{align*}\pi_G(g):{\mathcal A}_\infty&\to{\mathcal A}_\infty\\
\varphi&\mapsto\{x\mapsto\varphi(g^{-1}x)\}\end{align*}
If $p:G\to G/H$ denotes the quotient map, we further define
$${\mathcal A}_0=\{\varphi\in{\mathcal A}_\infty\hbox{ and }  |p(\hbox{supp}\varphi)|<\infty\}$$
and check that the induced representation $\pi_G$ stabilizes
${\mathcal A}_0$. Given any norm on $V$, any element
$\varphi\in{\mathcal A}_\infty$ induces a well-defined 
map $$N_\varphi: \left\{\begin{array}{ccc} G/H &  \to & \mathbb{R} \\  gH &
                                                                    \mapsto &  \|\varphi(g)\|. \end{array}    \right.$$ 
        This     map is  indeed   well-defined by       definition of
        $\mathcal{A}_\infty$, and because   $\pi_H$ is a representation
        in the   isometry group. If the norm on $V$ is an $\ell^p$
        norm, then this map induces a norm on 
        $\mathcal{A}_0$, by the formula 
        $$N(\varphi)= \left( \sum_{G/H} N_\varphi (gH)^p \right)^{1/p}$$  
        that is an
        $\ell^p$-norm. One then 
        defines ${\mathcal A}_p$ to be the $\ell_p$ completion of
        ${\mathcal A}_0$ with respect to $N$.
\end{defi}
Notice that for all $\varphi \in \mathcal{A}_\infty$, for all $g \in G$, the value of $\varphi$ at $g$ determines the
restriction of $\varphi$ on $gH$. Thus  any system of representatives for $G/H$ gives an
identification of ${\mathcal A}_0$ with finitely $V$-valued supported
functions on $G/H$, and ${\mathcal A}_p$ is the inverse image of
$\ell^p(G/H,V)$ under this identification. If the space $V$ we started with is uniformly convex, then so will ${\mathcal A}_p$ be according to the above cited result of Day \cite{Day}.
When $G$ is a relatively hyperbolic group, and $H$ is a peripheral subgroup that
acts properly on an $\ell_p$-space, we
will use this induced representation to produce a action of $G$ on an
$\ell_p$-space that is proper in a sense of angles.  To do that, we
will need to produce a cocycle of $G$ for this induced
representation.
\subsection{The example of a free product}\label{sec;freeprod}
In this subsection we
are  given a group $H$,  a unitary representation $(\pi_H,V)$ of $H$ on a space $V$,
and a cocycle $c:H\to V$.   We consider  $G= K*H$ for an arbitrary
group $K$. We want to produce a cocycle $C$ for $G$ in
the space $\mathcal{A}_0$ for the induced representation (hence an
affine isometric action of $G$ on $\mathcal{A}_0$). This is a well-known construction but we present it here in the way needed for its generalization to relatively hyperbolic groups.

We consider $\widetilde{G/H} \subseteq G$, the coset representatives of
$H$ in $G$ given by normal forms in the free product.
More precisely,  any element $g\in G$ is uniquely of the form
$g=k_1h_1\dots k_nh_n$ with $k_i\in K$, $h_i\in H$ and
$k_2,\dots,k_n,h_1,\dots h_{n-1}\not=e$.  For a coset $gH$,  we  write
$g=k_1h_1\dots k_nh_n$, and we choose the representative $\tilde{g} =
k_1h_1\dots k_n$. Note that if $g'\in gH$, then $\tilde{g'} =
\tilde{g}$, as it should be. One can picture that
$\tilde{g}$ corresponds to the projection of $1$ on the coset $gH$ for
any word metric adapted to the free product.  We define, for
each $\gamma\in G$ the vector $C_\gamma \in \mathcal{A}_0$ as follows
$$ C_\gamma : 
\left\{   \begin{array}{ccc} G & \to &  V   \\ 
g &            \mapsto  &           c(g^{-1}\widetilde{g})  -
                          c(g^{-1}\gamma (\widetilde{\gamma^{-1} g})).
                                                 \end{array}   \right.  $$
Let us first check that for all $\gamma$, $C_\gamma$ is in  $\mathcal{A}_\infty$.
Let us compute  $C_\gamma(gh)$ for arbitrary $g\in G, h\in H$. In the following computation, the first line holds because $ (\widetilde{ gh})$ and  $(\widetilde{ g})$ are equal since they are both the representative of the coset $gH$, and the second line holds by the cocycle relation for $c$. 
\begin{eqnarray*} C_\gamma(gh) & = & c(h^{-1} g^{-1}  (\widetilde{ gh}) )-   c(h^{-1} g^{-1} \gamma (\widetilde{ \gamma^{-1} gh}))  \\
 & = &  c(h^{-1} g^{-1}  (\widetilde{ g})) -   c(h^{-1} g^{-1} \gamma (\widetilde{ \gamma^{-1} g}) ) \\
 & = & c(h^{-1}) + \pi_H(h^{-1}) c(  g^{-1}  (\widetilde{ g})) - c(h^{-1}) - \pi_H(h^{-1}) c( g^{-1} \gamma (\widetilde{ \gamma^{-1} g}) ) \\
 & = & \pi_H(h^{-1}) C_\gamma(g).
\end{eqnarray*}
Therefore,  for all $\gamma$, $C_\gamma$ is in  $\mathcal{A}_\infty$.

Geometrically, for any fixed $g$,   this records the difference  (seen through the cocycle $c$)  of the     
   projection of $1$     and the projection     of $\gamma$, on
   the coset $gH$ (see Figure \ref{fig;1} for an illustration).
\begin{figure}\label{fig;1}
\includegraphics[width=0.7\textwidth]{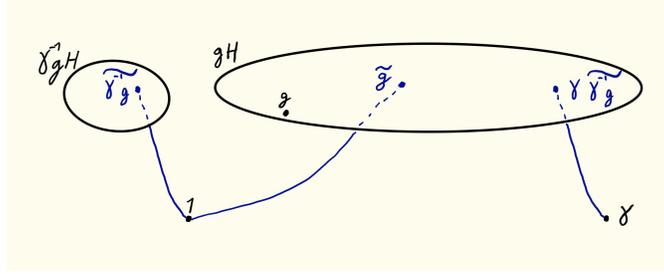}
\caption{Cosets seen from $1$ and from $\gamma$}
\end{figure}
      Since we are in a free product, only finitely many such
     differences are non-zero, and hence the map $\{\gamma \mapsto C_\gamma\}$
    takes its values in $\mathcal{A}_0$. 
    
    Now we check that $C$ is a cocycle of $G$ for the induced
    representation. Let us introduce $d: G\to V$ by $d(g)=c(g^{-1}
    \tilde{g})$. Then, for $\gamma \in G$, the map $\pi_G(\gamma)d :
    G\to V$  is, by definition of the representation, so that
    $(\pi_G(\gamma)d)(g) = d(\gamma^{-1} g)$  for all $g\in G$. This
    means that $(\pi_G(\gamma)d)(g) =  c(g^{-1} \gamma
    \tilde{\gamma^{-1} g})$, and we see that $C_\gamma = d-
    \pi_G(\gamma)d$ for all $\gamma$. The representation $\pi_G$ is a
    homomorphism, therefore, this immediately implies the cocycle
    relation $ C_{\gamma_1\gamma_2} = C_{\gamma_1} +\pi_G(\gamma_1)
    C_{\gamma_2}$.   Actually, one can see $C$ as the coboundary of
    $d$ in the space $\mathcal{A}_\infty  $, however, it is no longer a coboundary 
    in $\mathcal{A}_0$, but the cocycle
    relation still holds.
    
Notice that the subspace $\mathcal{A}_H$ of
                          $\mathcal{A}_0$ consisting of functions
                          $\varphi$ whose support is contained in $H$,
                         is isomorphic with $V$, 
                         and on this subspace $\pi_G$ and $C$ coincide with $\pi_H$ and
                          $c$.    We thus see that one can extend any
                          affine isometric action of $H$ on an $\ell^p$-space to an
                          action of  $K*H$ on an $\ell^p$-space, such
                          that the restrictions to each conjugate of
                          $H$ are all different and all isomorphic to
                          the initial action of $H$.


\subsection{Random representatives}
When $G$ is a relatively hyperbolic group, and $H$ is a peripheral
                          subgroup, we don't have in general
                          sufficiently stable coset representatives
                          for $H$ to argue like in the free product case. 
 However, as we will see,  the flow of the previous section allows
 to define, for each coset, a probability measure whose support is in
 the coset, which we think of as a random representative of the
 coset, that enjoy an $\ell^p$ version of the stability of the
 canonical normal forms of free products. We develop here the needed
 vocabulary on random representatives.

\begin{defi}\label{randomCoset}Given a discrete group $G$ and a subgroup $H<G$, a {\it random set of representatives} for $G/H$ is a section 
\begin{align*}\nu:G/H&\to\Proba(G)\\
gH&\mapsto\nu^{gH}\end{align*}
for the map $p:\Proba(G)\to{\mathcal P}(G/H)$ that assigns to a probability measure on $G$, the projection of its support in $G/H$. We will say that random set of representatives $\nu$ is {\it finite} if all the measures $\nu^{gH}$ have finite support, and {\it uniformly finite} if all those supports have cardinality uniformly bounded.\end{defi}
Notice that a set of representatives in the usual sense is a section whose image consists of Dirac masses and that for any coset $gH\in G/H$, the probability measure $\nu^{gH}$ is supported on the coset $gH\subseteq G$.
\begin{defi} For a group $G$ and a subgroup $H$, a random set of representatives $\nu$ will be {\it almost $G$-invariant} if for any $\gamma\in G$, then 
$$\nu^{gH}=\gamma\nu^{\gamma^{-1} gH}$$
except for finitely many cosets $gH\in G/H$ at most (see Figure \ref{fig;2} for an illustration).
\end{defi}
\begin{rems} \label{rem;randomreps} \begin{enumerate} \item For a group $G$ and a subgroup $H$, a random set of representatives $\nu$ will in general not be $G$-invariant. Indeed, $G$-invariance means that for any $\gamma\in G$ and any coset $gH\in G/H$, any $x\in gH$ one has
$$\nu^{gH}(x)=\nu^{\gamma^{-1} gH}(\gamma^{-1}x)$$
so in particular for $\gamma\in H$ and the coset $gH=H$, for any $x\in H$ we have $\nu^H(x)=\nu^H(\gamma^{-1}x)$, meaning that the measure $\nu^H$ is an $H$-invariant probability measure on $H$, forcing $H$ to be finite and $\nu^H$ to be the uniform measure on $H$.

\item If $G=K*H$ is a free product, then the  canonical set of
  representatives of $H$-cosets, given by the normal form (and used in
  \ref{sec;freeprod}) gives a random set of representatives
  $\nu$ that assigns to each coset, the Dirac mass of the canonical
  representative. This is, as it should be,  an almost $G$-invariant
  random set of representatives. Indeed, for any $\gamma\in G$
  and any coset $gH$, if we assume that $g$ is the canonical
  representative, then $\gamma^{-1}g$ is the canonical representative
  of $\gamma^{-1}gH$ unless there are cancellations on the normal
  forms leaving an element of $H$ in the end position. Those cosets
  $gH$ are vertices  on the geodesic in the Bass-Serre tree between
  $e$ (the base edge of the tree)  and $\gamma e$, so there are only
  finitely many of those.

\item\label{sc} If $G$ is a group obtained by a small cancellation $C'(\lambda)$
  quotient over
  a free product $K*H$ (for $\lambda <<1$),  consider a generating set
  consisting of generators of $H$ and of $K$, and the family $(R_i)$
  of relators satisfying the small cancellation condition.  For each left coset $gH$ of
  $H$ we consider a geodesic  from $1$ to   $gH$, and denote by $g_0\in
  gH$ its end point.  We distinguish
  two cases. If the geodesic has a final subsegment labelled by a word $\sigma$
  appearing in a relation $R_i$ or $R_i^{-1}$,  of lenght at least
  $2\lambda\times |R_i|$, then
  write, up to cyclic permutation $R_i=
  \sigma h R'_i$  (or $R_i^{-1} = \sigma h R'_i$) where $h\in H$ and the first letter of $R'_i$ is in
  $K$, and     we
  delare that the random representative of $gH$ consists of two dirac
  masses of weight $\frac12$ on the elements $g_0$ which is the end point
  of the geodesic,  and $g_0h$.  If there is no such segment $\sigma$
  on our chosen  geodesic  from $1$ to   $gH$, then the random
  representative of $gH$ is $g_0$ with probability $1$. One can prove,
  using standard small cancellation and hyperbolicity argument (like
  in \cite[Appendix]{RS95}), that
  in the later case, all geodesics from $1$ to $gH$  (and more
  generally all geodesics from any point to $gH$, provided they fellow
  travel $[1,gH]$ for a sufficient length) must enter $gH$ at
  $g_0$. In the former case, by similar arguments, these geodesics must enter $gH$ either on
  $g_0$ or on $g_0h$. For larger $\lambda$, up to $\lambda <1/6$, this is a consequence of \cite[Prop. 3.6, and Rem. 3.8]{GS}.\end{enumerate}
\end{rems}
\begin{defi}\label{lpalmost} For a group $G$ and a subgroup $H$, a random set of representatives $\nu$ will be {\it $\ell^p$-almost $G$-invariant} if for any $\gamma\in G$, the map
$$gH \mapsto  \|\nu^{gH}-\gamma\nu^{\gamma^{-1}gH}\|_1$$ 
belongs to $\ell_p(G/H)$.
\end{defi}
\subsection{Cocycle from random representatives} In this subsection we will see how, given a random set of representatives for a subgroup $H$ in a larger group $G$, we can construct a cocycle for the induced representation, starting from a cocycle on $H$. In case where the random set of representatives is $\ell^p$-almost $G$-invariant, we will see that the induced cocycle is actually on the $\ell^p$ induced representation.
\begin{lemma}\label{InducedCocycle}
Let $G$ be a group and let $H$ be a subgroup of $G$, with a unitary representation $(\pi_H,V)$ of $H$ and $c:H\to V$ a cocycle for $(\pi_H,V)$. Then for any finite random set of representatives $\nu$, the map
defined on $G$ by $C:\gamma\mapsto C_\gamma$ through the formula
$$C_\gamma(g)=\sum_{x\in gH}\nu^{gH}(x)c(g^{-1}x)-\sum_{y\in gH}\nu^{\gamma^{-1} gH}( \gamma^{-1}y)c(g^{-1} y)$$
is a cocyle on ${\mathcal A}_\infty$ for the induced representation.
\end{lemma}
\begin{figure}\label{fig;2}
\includegraphics[width=0.7\textwidth]{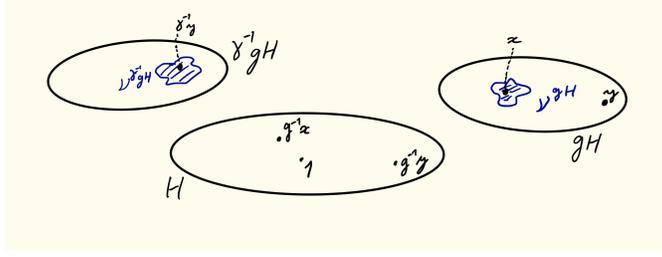}
\caption{Probability measures for coset representatives}
\end{figure}

\begin{proof}Since we assumed the random set of representatives to be finite, $C_\gamma(g)$ is well-defined by a finite sum. We first check that $C_\gamma(gh)=\pi_H(h^{-1})C_\gamma(g)$. This is a straightforward computation using first that $gH=ghH$ (hence  $\gamma^{-1} gH=\gamma^{-1} ghH$), then the cocycle condition on $c$. Indeed, one checks
\begin{eqnarray*}C_\gamma(gh)&=&\sum_{x\in gH}\nu^{gH}(x)c(h^{-1}g^{-1}x)-\sum_{y\in gH}\nu^{\gamma^{-1} gH}(\gamma^{-1}y)c(h^{-1}g^{-1}y)\\
&=&\sum_{x\in gH}\nu^{gH}(x)(c(h^{-1})+\pi_H(h^{-1})c(g^{-1}x))\\
& &- \sum_{y\in gH}\nu^{\gamma^{-1} gH}(\gamma^{-1}y)(c(h^{-1})+\pi_H(h^{-1})c(g^{-1} y))\\
&=&\sum_{x\in gH}\nu^{gH}(x)\pi_H(h^{-1})c(g^{-1}x)-\sum_{y\in gH}\nu^{\gamma^{-1} gH}(\gamma^{-1}y)\pi_H(h^{-1})c(g^{-1}y)\\
&=&\pi_H(h^{-1})C_\gamma(g)
\end{eqnarray*}
Now let us check the cocycle condition, namely that $C_{\gamma_1\gamma_2}=C_{\gamma_1}+\pi_G(\gamma_1)C_{\gamma_2}$. 
We define
$$d: G\to V, g\mapsto d(g)= \sum_{x\in gH}\nu^{gH}(x)c(g^{-1}x),$$
and we prove that
$C_{\gamma} = d - \pi_G(\gamma) d$ (which immediately implies the cocycle
relation).
Let us compute $\pi_G(\gamma)d$. By definition of $\pi_G$, one has $\pi_G(\gamma)d (g)= d(\gamma^{-1} g)$, which means  $\pi_G(\gamma)d (g)= \sum_{x\in \gamma^{-1} gH}\nu^{\gamma^{-1}gH}(x)c(g^{-1}\gamma x) $. Using the change of variable $ y=\gamma x$, one gets $\pi_G(\gamma)d (g)= \sum_{y\in  gH}\nu^{\gamma^{-1}gH}(\gamma^{-1} y)c(g^{-1}y) $. Thus,  $C_{\gamma} = d - \pi_G(\gamma) d$ and $C$ is a cocycle.

\end{proof}
\begin{rem}If the random set of representatives $\nu$ is uniformly finite and almost $G$-invariant, then the cocycle defined in Lemma \ref{InducedCocycle} is a cocycle on ${\mathcal A}_0$ and on ${\mathcal A}_p$ for any $p\geq 1$. Indeed, let us check that for any $\gamma\in G$ then $C_\gamma\in{\mathcal A}_0$. Namely, we have to check that the map $\{gH\mapsto \|C_\gamma(g)\|\}$ has finite support for any fixed $\gamma$. Except for the finitely many cosets where $\nu^{gH}\not=\gamma\nu^{\gamma^{-1} gH}$, we have that for $y\in gH$
$$\nu^{\gamma^{-1}gH}(\gamma^{-1}y)c(g^{-1} y)=\nu^{gH}(y)c(g^{-1} y)$$
so we get that $C_\gamma(g) =0$. It is in general too optimistic to
ask for almost $G$-invariance, this happens for instance in the case
where $G=H*K$ or $G$ is hyperbolic relative to $H$ and has a small
cancellation property.\end{rem}
\begin{lemma}\label{pcocycle}
Let $G$ be a group and let $H$ be a subgroup of $G$, with a unitary representation $(\pi,V)$
of $H$ and $c:H\to V$ a cocycle for $(\pi,V)$. Then for any uniformly finite $\ell^p$-almost $G$-invariant random set of representatives $\nu$, the map defined on $G$ by $C:\gamma\mapsto C_\gamma$ through the formula
$$C_\gamma(g)=\sum_{x\in gH}\nu^{gH}(x)c(g^{-1}x)-\sum_{y\in gH}\nu^{\gamma^{-1}gH}(\gamma^{-1}y)c(g^{-1}y)$$
is a cocyle on ${\mathcal A}_p$ for the induced representation.
\end{lemma}
\begin{proof} We already know from Lemma \ref{InducedCocycle} that it is a cocycle on ${\mathcal A}_\infty$, so it remains only to check that for any $\gamma\in G$ then $C_\gamma\in{\mathcal A}_p$. Namely, we have to check that the map $\{gH\mapsto \|C_\gamma(g)\|\}$ is $p$-summable for any fixed $\gamma$. Because $\nu$ is $\ell^p$-almost $G$-invariant, there are only finitely many cosets where $\hbox{supp}(\nu^{gH})\cap\hbox{supp}(\gamma\nu^{\gamma^{-1}gH})=\emptyset$. For the other cosets, we decompose $\nu^{gH}=\eta+\nu_1^0$ and $\nu^{\gamma^{-1}gH}=\eta+\nu_\gamma^0$ for $\eta={\mathcal M}\{\nu^{gH},\gamma\nu^{\gamma^{-1}gH}\}$, so that 
\begin{eqnarray*}C_\gamma(g)&=&\sum_{x\in gH}(\eta(x)+\nu_1^0(x))c(g^{-1}x)-\sum_{y\in gH}(\eta(y)+\nu_\gamma^0(y))c(g^{-1}y)\\
&=&\sum_{x\in gH}\nu_1^0(x)c(g^{-1}x)-\sum_{y\in gH}\nu_\gamma^0(y)c(g^{-1}y)\end{eqnarray*}
Since $C_\gamma\in{\mathcal A}_\infty$, without loss of generality we can assume that $g\in\hbox{supp}(\nu^{gH})\cap\hbox{supp}(\gamma\nu^{\gamma^{-1}gH})\not=\emptyset$, so that 
$$\max\{\|c(g^{-1}x)\|\,|\,x\in\hbox{supp}(\nu^{gH}\cup\gamma\nu^{\gamma^{-1}gH})\}\leq\max\{\|c(z)\|\,|\,z\in B(e,2R)\subseteq H\}= M,$$ 
where $R$ is the uniform bound on the supports of $\nu$ and hence $M$ does not depend on $g$ or $\gamma$. Since both $\nu_1^0$ and $\nu_\gamma^0$ are bounded by $\|\nu^{gH}- \gamma\nu^{\gamma^{-1}gH}\|_1$, we deduce that $\|C_\gamma(g)\|\leq 2M\|\nu^{gH}- \gamma\nu^{\gamma^{-1}gH}\|_1$, which allows us to conclude.
\end{proof}

\subsubsection{Contribution and $H$-properness} \label{sec;contrib}

We define now the contribution of a coset of $H$ in the value at $\gamma$ of the cocycle $C$.  Given two coset representatives,  $x$  for $gH$ and $x'$  for $\gamma^{-1} gH$, the contribution (for these coset representatives) of $gH$ in $C_\gamma$ will be the word distance in $H$ between  $g^{-1}x$ and  
    $g^{-1}\gamma x'$. It may help to think of these $x$ and $\gamma x'$ as projections on $gH$ of $1$ and of $\gamma$. In the context of radom coset representatives, we will sum over the probability measures.

\begin{defi}Let $G$ be a group, $H$ be a subgroup of $G$ and let $\nu$ be a random set of representatives for $H$ in $G$. Consider a word metric $d_H$ on $H$.
For each $\gamma\in G$, let us define the {\it $(d_H)$-contribution for $\gamma$} with
respect to the random set of representatives $\nu$,  to be 
$$ \max_{gH \in G/H}   \,  \left( \sum_{(x,y)\in (gH)^2}  \left(\nu^{gH} (x)
  \nu^{\gamma^{-1} gH} (\gamma^{-1}y)\right)    d_H( g^{-1}x, 
    g^{-1} y)  \right) \,  \qquad {\rm (Contrib)}$$ 
Observe that in each term of the maximum, the only $x$ and $y$ that can be of
interest, are those respectively in the support of $\nu^{gH}$ and in
the $\gamma$-translate of the support of  $\nu^{\gamma^{-1}gH}$. 
We say that a coset $gH$ is {\it involved in the contribution} if it realises the maximum in the formula.
\end{defi}
\begin{rem} Since we will be interested in whether or not the contribution tends to infinity (for a sequence of elements), the choice of $d_H$ among possible word metrics in $H$ is often irrelevant, and we simply talk about the contribution without precising $d_H$.
\end{rem}

\begin{rem}If our random set of representatives consists of
Dirac masses at $\tilde{g}$ (for the coset $gH$), as it is the case
for the free products, then the $(d_H)$-contribution is
$$ \max_{G/H}  d_H( \tilde{g}, \gamma (\widetilde{\gamma^{-1} g})).$$ 
In the case of a free product, this is the maximum of the word lengths of the
elements $h_i$ of the normal form of $\gamma$. \end{rem}

\begin{defi}We say that the cocycle $C$ is {\it $H$-proper} if   
  $\|C_{\gamma}\|$ goes to infinity with the $(d_H)$-contribution  of $\gamma$.
\end{defi}

Using the geometric idea of projection, it means that if there is a coset of $H$ on which the projection of $\gamma$ is far from the projection of $1$, then $C_\gamma$ is a large vector.
We can now conclude with the main result for this section
\begin{prop} \label{prop;induced_affine_rep} Let $G$ be a group and
  let $H$ be a subgroup of $G$, and a unitary representation $(\pi,V)$
  of $H$ and $c:H\to V$ a proper cocycle for $(\pi,V)$. Then for any
  finite $\ell^p$-almost $G$-invariant random set of representatives $\nu$ whose supports are
  uniformly bounded, the ${\mathcal A}_p$ cocycle for the induced representation defined in Lemma \ref{pcocycle} is $H$-proper.
\end{prop}
\begin{proof} Let $D$ be a bound on the diameter of the supports of $\nu^{gH}$. Let
 $D_c\geq 1$ be a bound on $\|c(h)\|$ for $h\in H$ such that $d_H(1,h) \leq
 2D$. 

By properness of $c$, we can choose $R>2D$ such that $\|c(h)\| >5D_c$ for each $h$ outside the ball of
radius $R$ in $(H, d_H)$.  

Because $\nu^{gH}$ is a probability, we can rewrite the sum defining $C_\gamma
(g)$ (in Lemma \ref{pcocycle}) as 
$$C_\gamma(g)=\sum_{(x,y) \in (gH)^2} \left(  \nu^{gH}(x)
  \nu^{\gamma^{-1} gH} (\gamma^{-1}y ) \right) \left( c(g^{-1}x) -
  c(g^{-1}y) \right)$$
which, by cocycle identity for $c$ gives
 $$C_\gamma(g) = \sum_{(x,y) \in (gH)^2} \left(  \nu^{gH}(x)
  \nu^{\gamma^{-1} gH} (\gamma^{-1}y ) \right) \left(   \pi(g^{-1} y)  c(y^{-1}x) \right).    $$
If $(x,y)$ and $(x',y')$ are  in $(gH)^2$, then by the cocycle relation we have
$$ \pi(g^{-1} y)   c({y}^{-1}x)  =\pi(g^{-1} y) c({y}^{-1}y') + \pi({g}^{-1}y') c({y'}^{-1} x') +
\pi( {g}^{-1}x'  ) c({x'}^{-1}x).$$
Let $g$ such that $gH$ is involved in the contribution for $\gamma$. We want to bound from below the norm of $C_\gamma$ in terms of the contribution for $\gamma$. Thus, 
we may assume that the diameter of the union of the supports of $\nu^{gH}$
and $\gamma \nu^{\gamma^{-1}gH}$ is
larger than $R$ (otherwise this contribution is less than $R$).  Take $x_0, y_0$ in those supports at maximal
distance from each other (hence larger
than $R$). The contribution of $gH$ for $\gamma$ is at least  $d_H( g^{-1}x_0, g^{-1}y_0) -2D  = d_H( 1, y_0^{-1} x_0)-2D$, and at most $d_H( 1, y_0^{-1} x_0)+2D$. Computing  $C_\gamma(g)$   we get
  \begin{align*}  C_\gamma(g) = \sum_{(x,y) \in (gH)^2} &\left(  \nu^{gH}(x)
  \nu^{\gamma^{-1} gH} (\gamma^{-1}y ) \right) \times \\
&\left(     \pi(g^{-1} y) c({y}^{-1}y_0) +\pi({g}^{-1}y_0) c(y_0^{-1} x_0) +
\pi( {g}^{-1}x_0  ) c(x_0^{-1}x)  
 \right). \end{align*} \\
Hence we can rewrite
 \begin{align*}C_\gamma(g) =      \pi({g}^{-1}y_0)
                    c(y_0^{-1} x_0)  & + \sum_{(x,y) \in (gH)^2}   \left(  \nu^{gH}(x)
  \nu^{\gamma^{-1} gH} (\gamma^{-1}y ) \right)  \left(  \pi(g^{-1} y)
                                              c({y}^{-1}y_0) \right)
  \\
    & +   \sum_{(x,y) \in (gH)^2}   \left(  \nu^{gH}(x)
  \nu^{\gamma^{-1} gH} (\gamma^{-1}y ) \right)  \left(   \pi( {g}^{-1}x_0  ) c(x_0^{-1}x)  
 \right). 
\end{align*}
 The first term $\pi({g}^{-1}y_0)
                    c(y_0^{-1} x_0)  $ has norm at least $5D_c$ by
                    assumption on $R$, and the result
of each sum has norm at most $D_c$ by triangular inequality (recall that the sum of the
coefficients is $1$) and assumption on $D_c$.  
  Therefore, in this case,  $\|C_\gamma(g) \|
\geq  \|c(y_0^{-1} x_0) \|   -2D_c \geq 3D_c$.  

Recall that the contribution  for $\gamma$ is between  $d_H( 1, y_0^{-1} x_0)-2D$ and $d_H( 1, y_0^{-1} x_0)+2D$.  By properness of $c$ the quantity  $ \|c(y_0^{-1} x_0) \|$ goes to infinity with the contribution for $\gamma$, and therefore so does  $\|C_\gamma(g) \|$.



\end{proof}

\subsection{Peripherally
  proper actions}

The flow of the previous section provides, in relatively
hyperbolic groups, random coset representatives for the
peripheral subgroups.
\begin{prop} \label{cor;periph_i} Let $G$ be a relatively hyperbolic group and let $H_1,\dots,H_k$ be the peripheral subgroups, and let $X$ be the coned-off Cayley graph of $G$ with respect to $H_1,\dots,H_k$. Then, for each $i=1,\dots,n$ the map
$$\nu:  \left\{ \begin{array}{ccc}  G/{H_i}& \to& \Proba(G) \\  gH_i&
                                                                      \mapsto&
                                                                               \mu_1(\widehat{gH_i}) \end{array}\right.$$
is an $\ell^p$-almost $G$-invariant random set of representatives for $H_i$, where $\mu_1(\widehat{gH_i})$ is the mask of the infinite valence vertex $\widehat{gH_i}$ from the identity $1$, as in Definition~\ref{flot}.

If moreover $H_i$ acts properly
  by affine isometries 
  on an $\ell^p$-space, then $G$ acts by affine isometries on an
  $\ell^p$-space, and this action is $H_i$-proper.
\end{prop}
\begin{proof} That $\nu^{gH_i}=\mu_1(\widehat{gH_i})$ is a probability measure
  supported on the coset $gH_i$ is the first assertion of Proposition
  \ref{prop;theta0}, hence $\nu$ is a random set of
  representatives for $H_i$. That this set is $\ell^p$-almost
  $G$-invariant is the content of Corollary
  \ref{cor;ALconique}. Indeed,
  $\mu_x(\widehat{gH_i})=  x\mu_1(\widehat{x^{-1} g H_i}) =  x\nu^{x^{-1}gH_i}$ by Proposition
  \ref{prop;mask_equivariance} (since the flow step is equivariant),
    hence $\| \mu_1(\widehat{gH_i}) \Delta
  \mu_{x}(\widehat{gH_i}) \|= \|\nu^{gH}-x\nu^{x^{-1}gH}\|_1$.   By Corollary
  \ref{cor;ALconique}, this is $\ell^p$-summable over the set 
  of left cosets of $H$.

Now, we can apply
Proposition \ref{prop;induced_affine_rep}, and obtain an action of $G$
by affine isometries on an  $\ell^p$-space that is $H_i$-proper.
Since angles at a vertex  $\widehat{gH_i}$ are bounded above by the word-length in $H_i$, we obtain the stated condition from the
definition of $H_i$-properness.\end{proof}
\section{Proper actions of relatively hyperbolic groups}
To finish the proof the our main result, Theorem \ref{main}, we start by a well-known remark.
\begin{rem}\label{BLM}If a group acts properly by affine isometries on an $\ell^p$-space, say $\ell^p(X)$ for $X$ a discrete set, then for all $q\geq p$, it also acts properly on $\ell^{q}(X)$. Indeed, Banach-Lamperti's theorem says that for $p\not=2$ the linear part comes from an action of $G$ on $X$, so yields a linear representation on $\ell^q(X)$ for any $q\geq 1$. Then we have that $\ell^p(X)\subseteq\ell^q(X)$ and hence the cocycle on $\ell^p$ can be used as is on $\ell^q$. Indeed, for any finite sequence $(a_i)$ then
$$(\sum_ia_i^q)^{\frac{1}{q}}=\left(\left(\sum_i(a_i^p)^{\frac{q}{p}}\right)^{\frac{p}{q}}\right)^{\frac{1}{p}}\leq (\sum_ia_i^p)^{\frac{1}{p}}$$
(the inequality being Minkowski inequality for $s=q/p\geq 1$) and hence on a discrete set $X$ we have that $\|\ \|_q\leq\|\ \|_p$, so that $\ell^p(X)\subseteq\ell^q(X)$.\end{rem}
\begin{proof}[Proof of Theorem \ref{main}] The first part of the theorem is the content of Theorem \ref{thm;actionAL}. Let $H_1,\dots,H_k$ be the peripheral subgroups of $G$ and assume that there are uniformly convex Banach spaces $B_i$ on which $H_i$ acts properly by affine isometries. We use, for each $i=1,\dots,n$, the action obtained in Proposition
\ref{cor;periph_i} and denote by $\mathscr{V}_i$ the uniformly convex Banach space obtained. We also
use Theorem \ref{thm;actionAL}, and denote by $\mathscr{V}_0$
the space obtained.  
Let $\mathscr{W} = \bigoplus_{i=0}^k \mathscr{V}_i$ be a uniformly convex product of the uniformly convex Banach spaces $\mathscr{V}_i$, for $i=0,\dots,k$. The action of $G$, coordinates by coordinates, is
by affine isometries and it remains to check that it is proper. This
amounts to checking that the action is proper in one of the
coordinates at least. Let $g_n$ be a sequence in $G$ going to
infinity. According to Proposition \ref{prop;pseudo_word_metric}, then
either $g_n$ goes to infinity in the coned-off graph as well, or
it remains bounded, and $[1,g_n]$ gets arbitrary large angles at vertices of infinite valence.    In
the first case $g_n(\vec{0})$ goes to infinity in the $\mathscr{V}_0$
coordinate. In the second case, $[1,g_n]$ has an arbitrarily large
angle   at some  vertex fixed by conjugates of $H_{i(n)}$. Extracting
so that $i(n)$ is constant (which can be done in a way that
partitions the sequence) 
by Proposition \ref{cor;periph_i}    $g_n(\vec{0})$ goes to infinity
in the corresponding coordinate  $\mathscr{V}_i$. In all cases,
$g_n(\vec{0})$ goes to infinity in $\mathscr{W} $.

If the peripheral subgroups act properly by affine isometries on some $\ell^p$, taking the maximum over all $p$'s on which the peripheral subgroups act, and the $p$ obtained in Theorem \ref{thm;actionAL}, we obtain $p$ large enough and (according to Remark \ref{BLM}) a proper action on the $\ell^p$-combination of those $\ell^p$-spaces.
\end{proof}
The following statement has an unfortunate technical assumption of compatibility between the CAT(0) cube complex and the system of random coset representatives, through the notion of contribution (see Section \ref{sec;contrib}). The proof is completely similar and in some sense a particular case of the proof of Theorem \ref{main}.
\begin{prop}\label{Haagerup}Let $G$ be a group acting on a CAT(0) cubical complex $X$, with finitely many orbits, and finite edge stabilizers. Assume that the stabilizer of any point has the Haagerup property (respectively, acts properly on and $\ell^p$-space) and admits a random system of coset representatives that is $\ell^2$-almost $G$-invariant. Assume also that for any sequence $(g_n)$ of elements of $G$ going to infinity in $G$, either for $x\in X$, $g_nx$, or the contribution for $g_n$ of cosets of stabilizers of vertices of $X$ goes to infinity. Then $G$ has the Haagerup property as well (respectively, $G$ acts properly on an $\ell^p$-space).\end{prop}
\begin{proof}
If we assume the existence of $\ell^2$-almost $G$-invariant random coset representatives we obtain, according to Proposition \ref{cor;periph_i} for each $i=1,\dots,k$, an action on a Hilbert space (respectively, an $\ell^p$-space) denoted $\mathscr{V}_i$. We also use Niblo-Reeves' construction \cite{NiRe} to obtain a Hilbert space, denoted by $\mathscr{V}_0$, from the action on a CAT(0) cube complex (and hence also on an $\ell^p$-space according to Remark \ref{BLM}). Let $\mathscr{W} = \bigoplus_{i=0}^k \mathscr{V}_i$ be the $\ell^2$ sum of the Hilbert spaces $\mathscr{V}_i$, for $i=0,\dots,k$, this is again a Hilbert space (respectively, the $\ell^p$-sum of the $\ell^p$-spaces $\mathscr{V}_i$, which is again an $\ell^p$-space). The action of $G$, coordinates by coordinates, is by affine isometries and it remains to check that it is proper. This amounts to checking that the action is proper in one of the coordinates at least. Let $g_n$ be a sequence in $G$ going to infinity. Assume first that there is a point $x$ in the CAT(0) cube complex (denoted by $X$)  so that $d(x,g_n(x))$ goes to infinity. Then $g_n(\vec{0})$ goes to infinity in the $\mathscr{V}_0$ coordinate. Assume now the other case, that 
$d(x,g_n(x))$ remains bounded for all points $x$.  Then by assumption, the contribution of some coset of vertex stabilizer goes to infinity. 
\end{proof}
\begin{proof}[Proof of Corollary \ref{scH}]
According to Remarks \ref{rem;randomreps} part (\ref{sc}), in finitely presented
small cancellation groups over a free product, the images of the factor
groups admit an almost $G$-invariant random system of coset representatives. Recall also
(see \cite{MS})  that any small cancellation
group over a free product acts on a cubical CAT(0) complex $X$ with
finitely many orbit, each of which has trivial stabilizer, and whose stabilizer of vertices are images of the free factors (and their conjugates).   It is
obtained, as in \cite{Wise}, by means of 
a wallspace structure  (or hypergraph system) of an associated
2-dimensional polygonal complex.  
 $G$ is hyperbolic relative to these stabilizers of vertices (see for instance \cite{Pankratev99}), and by \cite[Prop 4.13]{BowRH}, the $1$-squeleton of $X$ is equivariantly quasi-isometric to a coned-off Cayley graph for $G$ over these subgroups.
 If, for a sequence $g_n$, and a point $x\in X$, the sequence $g_nx$ is bounded, then the sequence $g_n$ is bounded in the coned-off Cayley graph. If it goes to infinity in the group, then the contribution for it of cosets of stabilizers of vertices goes to infinity.  
 Proposition \ref{Haagerup} hence applies.
\end{proof}

\smallskip

\noindent {\sc Indira Chatterji,  Laboratoire de Math\'ematiques J.A. Dieudonn\'e
UMR no. 7351 CNRS UNS Universit\'e de Nice-Sophia Antipolis 06108 Nice Cedex 02, France}

\noindent {\sc Fran\c{c}ois Dahmani, Institut Fourier, Univ. Grenoble
  Alpes, F-38000 Grenoble, France } 
\end{document}